\documentclass[leqno,11pt]{amsart}
\usepackage [cp1250]{inputenc}
\usepackage{amsmath}
\usepackage{amsfonts}

\usepackage{amsthm}
\usepackage{amssymb}
\usepackage{times}
\usepackage{geometry}
\usepackage{xfrac}
\usepackage{amssymb}
\usepackage{verbatim}
\usepackage{dsfont}

\usepackage[charter]{mathdesign}

\DeclareFontFamily{U}{dutchcal}{\skewchar\font=45 }
\DeclareFontShape{U}{dutchcal}{m}{n}{<-> s*[1.0] dutchcal-r}{}
\DeclareFontShape{U}{dutchcal}{b}{n}{<-> s*[1.0] dutchcal-b}{}
\DeclareMathAlphabet{\mathlcal}{U}{dutchcal}{m}{n}
\SetMathAlphabet{\mathlcal}{bold}{U}{dutchcal}{b}{n}

\usepackage{enumerate}
\usepackage{verbatim}
\usepackage{amsfonts}
\usepackage{graphicx}
\usepackage{tikz}
\usetikzlibrary{matrix}
\def\ocirc#1{\ifmmode\setbox0=\hbox{$#1$}\dimen0=\ht0
    \advance\dimen0 by1pt\rlap{\hbox to\wd0{\hss\raise\dimen0
    \hbox{\hskip.2em$\scriptscriptstyle\circ$}\hss}}#1\else
    {\accent"17 #1}\fi}

\newcommand{\mres}{\mathbin{\vrule height 1.6ex depth 0pt width
0.13ex\vrule height 0.13ex depth 0pt width 1.3ex}}
\newcommand{\R}{\mathbb{{R}}}

\newcommand{\rn}{{\mathbb{R}^n}}
\newcommand{\mup}{{\mu_{\vp(\cdot)}}}
\newcommand{\vr}{{\varrho}}
\newcommand{\ep}{{\epsilon}}
\newcommand{\ve}{{\varepsilon}}
\newcommand{\dv}{{\rm div\,}}
\newcommand{\xinm}{{\xi_n^{\rm meas}}}
\newcommand{\gam}{{\gamma^{\rm meas}}}
\newcommand{\etam}{{\eta^{\rm meas}}}
\newcommand{\Mb}{{\mathcal{M}_b}}
\newcommand{\muvp}{{\mu_{\vp(\cdot)}}}
\newcommand{\musp}{{\mu^+_{\rm sing}}}
\newcommand{\musm}{{\mu^-_{\rm sing}}}
\newcommand{\mus}{{\mu_{\rm sing}}}

\newcommand{\MP}{{\mathcal{M}^{\vp(\cdot)}_b}}
\newcommand{\capP}{{\mathrm{C}_{\vp(\cdot)}}}
\newcommand{\CapP}{{\mathrm{cap}_{\vp(\cdot)}}}

\newcommand{\e}{\varepsilon}
\newcommand{\vp}{\varphi}
\newcommand{\A}{\mathcal{A}}
\newcommand{\oppA}{{\mathfrak{A}_{\vp(\cdot)}}}

\usepackage{color}

\newcommand{\N}{\mathbb{N}}
\newcommand{\wt}{\widetilde}
\newcommand{\el}{{\mathlcal{l}}}

\theoremstyle{definition}

\newcommand{\snr}[1]{\lvert #1\rvert}
\newcommand{\nr}[1]{\lVert #1 \rVert}
\def\rp0{{[0,\infty)}}

\newtheorem{theorem}{Theorem}
\newtheorem{lemma}{Lemma}
\newtheorem{proposition}[lemma]{Proposition}
\newtheorem{corollary}[lemma]{Corollary}
\newtheorem{definition}[lemma]{Definition}
\newtheorem{remark}[lemma]{Remark}
\newtheorem{rem}[lemma]{\bf Remark}

\swapnumbers

\title{Measure data elliptic problems\\ with generalized Orlicz growth}
\author{Iwona Chlebicka}\address{Iwona Chlebicka\\Faculty of Mathematics, Informatics and Mechanics, University of Warsaw\\ul. Banacha 2, 02-097 Warsaw, Poland} \email{\texttt{i.chlebicka@mimuw.edu.pl}}
\begin{document}

\subjclass[2010]{35J60 (46E30)\vspace{1mm}} 

\keywords{ Capacity, Elliptic PDEs,  Measure data problems, Musielak--Orlicz spaces, Orlicz--Sobolev spaces, Very weak solutions\vspace{1mm}}

\thanks{{\it Acknowledgements.}\ I. Chlebicka is supported by NCN grant no. 2016/23/D/ST1/01072. 
\vspace{1mm}}

 \pagestyle{empty}
 \begin{abstract} We study nonlinear measure data elliptic problems involving the operator exposing generalized Orlicz growth. Our framework embraces reflexive Orlicz spaces, as well as natural variants of variable exponent and double-phase spaces. Approximable and renormalized solutions are proven to exist and coincide for arbitrary measure datum and to be unique when the datum is diffuse with respect to a relevant nonstandard capacity. For justifying that the class of measures is natural, a capacitary characterization of diffuse measures is provided. 
 \end{abstract}

\maketitle

\setcounter{tocdepth}{1}

\section{Introduction}

 Our objective is to study existence and uniqueness of two kinds of very weak solutions to nonlinear measure data problem
\begin{equation}
\label{eq:main}
\begin{cases}-\dv \A(x,\nabla u) =\mu & \text{in }\ \Omega,\\
u=0 &  \text{on }\ \partial\Omega,\end{cases}
\end{equation}
where  $\Omega\subset\rn$ is a bounded, $n\geq 2$,  $\mu$ is an arbitrary bounded measure on $\Omega$, and $\A:\Omega\times\rn\to\rn$ has growth prescribed be the means of an inhomogeneous function $\vp:\Omega\times\rp0\to\rp0$ of an Orlicz growth with respect to the second variable. Special cases of the leading part of the operator $\A$ include $p$-Laplacian, $p(x)$-Laplacian, but we cover operators with Orlicz, double-phase growth, as well as weighted Orlicz or variable exponent double phase one as long as it falls into the realm of Musielak-Orlicz spaces within the natural regime described in Section~\ref{sec:prelim}.  The existence of renormalized solutions to general measure data problem and uniqueness for diffuse measures is new even in the reflexive Orlicz case. It was also not known in two cases enjoying lately particular attention -- double-phase and variable exponent double phase ones.

Very weak solutions to measure-data problems of the  form~\eqref{eq:main} are already studied in depth in the classical setting of Sobolev spaces, that is when the growth growth of the leading part of the operator is governed by a power function with the celebrated special case of $p$-Laplacian $\Delta_p u=\dv(|\nabla u|^{p-2}\nabla u)$. To give a flavour let us mention e.g.~\cite{BBGGPV,BGO,BM,DaMMa,DMMOP1}, where the existence is provided for various notions of very weak solutions for $L^1$ or measure data. Note that the notions in many cases coincide~\cite{kikutu,DaMMa,DMMOP1}. In general it is possible to find a~proper notion enjoying existence, but sharp assumptions on $\mu$ to ensure uniqueness for these type of problems are not known even when the operator $\A$ exposes the mentioned standard $p$-growth. See counterexamples in~\cite{BGO} on non-uniqueness for concentrated measures.  The natural sufficient condition in the standard case is that $\mu$ is so diffuse that it does not charge the sets of proper capacity zero and the proof of uniqueness essentially employs its characterisation of the form of Theorem~\ref{theo:decomp}.

Analysis of problems exposing $(p,q)$-growth, where the operator is trapped between polynomials $|\xi|^p\lesssim \A(x,\xi)\cdot\xi\lesssim 1+|\xi|^q,$ are already classical topic investigated since~\cite{Gossez-pde,m1,Talenti}. Nowadays, there is a great interest in analysis under nonstandard growth conditions that embraces more: problems with variable exponent growth used in~modelling of electrorheological fluids~\cite{mingione02,el-rh2}, thermistor model~\cite{zhikov9798} or image processing~\cite{chen06}, with double-phase growth good for description of composite materials~\cite{comi}, as well as Orlicz one -- engaged in modelling of non-Newtonian fluids~\cite{gwiazda-non-newt} and elasticity~\cite{ball}. Studies on nonstandard growth problems form a solid  stream in the modern nonlinear analysis~\cite{bacomi-st,comi,demi,deoh,IC-pocket,ChDF,CZG-gOp,hahale,hahato,Marc2020}. The theory of existence of very weak solutions to problems with nonstandard growth and merely integrable data is under intensive investigation~\cite{ab,BW,CiMa,gw-e,pgisazg1,San-Ur,Zhang}. For the study on Musielak-Orlicz-growth $L^1$-data elliptic equations we refer to~\cite{gw-e} under growth restrictions on the conjugate of the modular function and to~\cite{pgisazg1}, where existence is provided either in (all) reflexive spaces or when the growth of modular function is well-balanced (and the smooth functions are modularly dense, cf. also~\cite{yags}). Analogous parabolic study can be found in~\cite{CGWKSG,icpgazg,t2}. For measure data problems with Orlicz growth to our best knowledge we can refer only to~\cite{ABEY} for some class of measures, \cite{BeBe,CiMa} for general measures in the~reflexive case extended in~\cite{ACCZG,CGZG}. In~\cite{ACCZG,CGZG,CiMa} besides existence also regularity in the scale of Marcinkiewicz-type spaces is provided even for solutions to measure data problems, but therein the uniqueness is obtained only if the datum is integrable. On the other hand, existence of very weak solutions and uniqueness in the case of diffuse measures is studied in the variable exponent setting in~\cite{Zhang,LvLiZou}. Here two kinds of very weak solutions are proven to exist and coincide for arbitrary measure datum.

 We consider~\eqref{eq:main} involving the leading part of the operator governed by a function $\vp:\Omega\times\rp0\to\rp0$ and, thereby, placing our analysis in an unconventional functional setting, where the norm is defined by the means of the functional 
\begin{flalign}\label{phix}
w\mapsto \int_{\Omega}\varphi(x,|Dw|) \,dx,
\end{flalign}  Let us make an overview of the special cases of the functional framework we capture. The   operator can be governed by power function variable in space, namely $\vp(x,s)=|s|^{p(x)}$, where $p:\Omega\to(1,\infty)$ is log-H\"older continuous, cf.~\cite{CUF}. Another model example we cover are non-uniformly elliptic problems living in spaces with the double phase energy, $\vp(x,s)=|s|^{p}+a(x)|s|^{q}$, where $a\in C^{0,\alpha}(\Omega)$ is nonegative and can vanish in some regions of $\Omega\subset\rn$, while exponents satisfy $1<p\leq q<\infty$ and are close in the sense that $\frac{q}{p}\leq 1+\frac{\alpha}{n}$ necessary for density of smooth functions~\cite{comi,yags}. What is more, we admit problems posed in the reflexive Orlicz setting, when $\vp$ is a~doubling $N$-function $\varphi(x,s)=\vp(s)\in\Delta_2\cap\nabla_2$, including Zygmund-type spaces where $\varphi_{p,\alpha}(s)=s^p\log^\alpha(1+s),$ $p>1,$ $\alpha\in\R$ or compositions and multiplications of functions from the family $\{\varphi_{\bar p,\bar \alpha}\}_{\bar{p},\bar{a}}$ with various parameters. More generally, under certain nondegeneracy and continuity conditions, given as (A0)-(A2) in Section~\ref{sec:prelim}, we capture also general case~\eqref{phix}. The remaining examples we can give here cover all weighted  reflexive Orlicz functionals with non-degenerating weights,  double phase functions with variable exponents $\vp(x,s)=|s|^{p(x)}+a(x)|s|^{q(x)}$, double phase with Orlicz phases $\vp(x,s)=\vp_1(s)+a(x)\vp_2(s)$ or multi-phase cases $\vp(x,s)=\sum_{i=1}^k a_i(x)\vp_i(s)$ (with appropriately regular weights) as long as conditions (A0)--(A2) are satisfied.  We refer to~\cite{IC-pocket} for a more detailed overview of differential equations and~\cite{hahabook} for the fundamental properties of the functional framework.

\subsection*{ Diffuse measures } The natural property of a measure to ensure uniqueness of very weak solutions to~\eqref{eq:main} is that $\mu$ is diffuse with respect to a relevant capacity.  In order to characterize such measures, let us denote  by $\Mb(\Omega)$ the set of bounded measures on $\Omega\subset \R^n$ and by $W^{1,\vp(\cdot)}(\Omega)$ the~Musielak-Orlicz-Sobolev space. See Section~\ref{ssec:spaces} for the introduction to the functional setting and all assumptions and Section~\ref{ssec:cap} for the capacity. By $\MP(\Omega)$ we mean the~set of $\vp(\cdot)$-diffuse measures (or $\vp(\cdot)$-soft measures) consisting of such bounded measures $\mup$ that do not charge sets of $\vp(\cdot)$-capacity zero (for every Borel set $E\subset\Omega$ such that $\capP(E)=0$ it holds that $\mup(E)=0$). One may think that a measure $\mup\in \MP(\Omega)$ is `absolutely continuous with respect to $\capP$'. Our first result is the following theorem.

\begin{theorem}[Characterization of measures]\label{theo:decomp} Suppose $\vp\in\Phi_c(\Omega)$ on a bounded domain $\Omega\subset\rn,$ $n\geq2$. Assume that $\vp$ satisfies (aInc)$_p$, (aDec)$_q$, (A0), (A1), and (A2). When $\mu\in\Mb(\Omega)$, then \[\mu_{\vp(\cdot)}\in\MP(\Omega)\quad\text{if and only if}\quad \mu_{\vp(\cdot)}\in L^1(\Omega)+(W^{1,\vp(\cdot)}_0(\Omega))',\] i.e. there exist
$f\in L^1(\Omega)$ and $G \in (L^{\wt\vp(\cdot)}(\Omega))^n$, such that $\mu_{\vp(\cdot)}=f-\dv G$ in the sense of~distributions. 
\end{theorem}

\begin{remark} Let us note that upon our assumptions $t^p\lesssim  \vp(\cdot,t)$. If $p>n$ it holds that \[\Mb(\Omega)\subset W^{-1,p'}(\Omega)\subset (W^{1,\vp(\cdot)}_0(\Omega))'.\] In this case all measures are absolutely continuous with respect to the Lebegue measure and, consequently, the result is really meaningful only for slowly growing functions $\vp$.\end{remark}

\begin{remark}\label{rem:decomp} The decomposition of Theorem~\ref{theo:decomp} \[\text{$\MP(\Omega)\ni\mu_{\vp(\cdot)}=f-\dv G$ with $f\in L^1(\Omega)$ and $G \in (L^{\wt\vp(\cdot)}(\Omega))^n$)}\] \underline{\em cannot be unique} as $ L^1(\Omega)\cap (W_0^{1,\vp(\cdot)}(\Omega))'\neq \{0\}.$ On the other hand, for every $\mu\in\Mb(\Omega)$ there exists a \underline{\em unique} decomposition 
\[\mu=\mu_{\vp(\cdot)} +\musp-\musm\]
with some $\mu_{\vp(\cdot)}$ which is absolutely continuous with respect to  ${\vp(\cdot)}$-capacity, while $\musp,\musm\geq 0$ are singular with respect to the $\capP$ (concentrated on some set of $\vp(\cdot)$-capacity zero), see Lemma~\ref{lem:basic-decomp}. 
Consequently, any $\mu\in\Mb(\Omega)$ admitts a decomposition
\[\mu=f+\dv G +\musp-\musm\]
in the sense of distributions, with some $ f\in L^1(\Omega)$, $ G \in (L^{\wt\vp(\cdot)}(\Omega))^n $, and $\capP$-singular $\musp,\musm$.
\end{remark}

Let us point out a consequence of Theorem~\ref{theo:decomp}, which to our best knowledge was not known in the classical  Orlicz-Sobolev spaces.
\begin{corollary}[Orlicz case] Suppose $B:[0,\infty)\to[0,\infty)$ is a Young function, such that $B\in\Delta_2\cap\nabla_2$. Then  $\mu_B\in \Mb(\Omega)$ does not charge the sets of Sobolev $B$-capacity zero if and only if $\mu_B\in L^1(\Omega)+(W^{1,B}_0(\Omega))'$, i.e. there exist $f\in L^1(\Omega)$ and $G \in (L^{\wt{B}}(\Omega))^n$, such that $\mu_B=f-\dv G$. In particular, the special case of this result is the classical measure characterization~\cite{BGO}: if $p>1$, then   $\mu_p\in \Mb(\Omega)$ does not charge the sets of the Sobolev $p$-capacity zero if and only if $\mu_p\in L^1(\Omega)+W^{-1,p'}(\Omega)$, i.e. there exist $f\in L^1(\Omega)$ and $G \in (L^{p'}(\Omega))^n$, such that $\mu_p=f-\dv G$ 
in the sense of distributions.
\end{corollary} 
We can also deduce from Theorem~\ref{theo:decomp} the decomposition in the variable exponent case.
\begin{corollary}[Variable exponent case] Suppose $p:\Omega\to(1,\infty)$ with $1<p_-\leq p(\cdot)\leq p_+<\infty$ is log-H\"older continuous and $p'(x):=p(x)/(p(x)-1)$. Then  $\mu_{p(\cdot)}\in \Mb(\Omega)$ does not charge the sets of Sobolev $p(\cdot)$-capacity zero if and only if $\mu_{p(\cdot)}\in L^1(\Omega)+W^{-1,p'(\cdot)}(\Omega)$, i.e. there exist $f\in L^1(\Omega)$ and $G \in (L^{p'(\cdot)}(\Omega))^n$, such that $\mu_{p(\cdot)}=f-\dv G$ 
in the sense of distributions, cf.~\cite{Zhang-Zhou}. Again, the special case  is the mentioned above classical power-growth case.
\end{corollary} 
More examples we can give here are decompositions in double phase spaces within the sharp range of powers, as well as some weighted Orlicz, variable exponent double phase, or multi-phase cases, under the prescribed natural non-degeneracy and continuity conditions.

\subsection*{Measure-data problems}

\subsubsection*{\underline{Assumptions}} Given $\vp\in\Phi_c(\Omega)$ on a bounded domain $\Omega\subset\rn,$ $n\geq2$, such that $\vp$ satisfies (aInc)$_p$, (aDec)$_q$, (A0), (A1), and (A2), we shall study equation~\eqref{eq:main} where vector field $\A$ satisfies the following conditions:\begin{itemize}
\item[$(\A 1)$] $\A:\Omega\times\rn\to\rn$ is Carath\'eodory function, i.e. it is measurable with respect to the first variable and continuous with respect to the last one;
\item[$(\A 2)$] There exist numbers $c_1^\vp,c_2^\vp >0$ and a function $0\leq \gamma\in L^{\wt\vp(\cdot)}(\Omega)$, such that for a.e. $x\in\Omega$ and all $\xi\in\rn$ the following ellipticity and growth conditions are satisfied
\[c_1^{\vp} \vp(x,|\xi|)\leq \A(x,\xi)\cdot\xi\quad\text{and}\quad |\A(x,\xi)|\leq c_2^\vp\left(1+\gamma(x)+\vp(x,|\xi|)/|\xi|\right).\] 
\item[$(\A 3)$]  $\A$ is monotone, i.e. for a.e. $x\in\Omega$ and all $\eta\neq \xi\in\rn$ 
\[\big(\A(x,\eta)-\A(x,\xi)\big)\cdot(\eta-\xi)>0.\]
\item[$(\A 4)$]  For a.e. $x\in\Omega$  it holds that $ \A(x,0)=0.$
\end{itemize} 

\subsubsection*{\underline{Special cases}}
\noindent Of course, $(\A 1)$--$(\A 4)$ with $\vp\in\Phi_c(\Omega)$ satisfying (aInc)$_p$, (aDec)$_q$, (A0), (A1), and (A2) embrace not only classical conditions in the case when $\vp(x,s)=s^p$:
\[c_1^p |\xi|^p\leq \A(x,\xi)\cdot\xi\quad\text{and}\quad |\A(x,\xi)|\leq c_2^p\left(1+\gamma(x)+|\xi|^{p-1}\right)\]
with $0\leq\gamma\in L^{p'}(\Omega)$ with the special case of (possibly weighted) $p$-Laplacian. When $\vp(x,s)=s^{p(x)}$ it covers 
\[c_1^{p(\cdot)} |\xi|^{p(x)}\leq \A(x,\xi)\cdot\xi\quad\text{and}\quad |\A(x,\xi)|\leq c_2^{p(\cdot)}\left(1+\gamma(x)+|\xi|^{p(x)-1}\right)\]
with  $0\leq\gamma\in L^{p(\cdot)/(p(\cdot)-1)}(\Omega)$ with the special case of (possibly weighted) $p(x)$-Laplacian. We allow for all $p:\Omega\to (1,\infty)$ under typical assumptions that $1<p_-\leq p(x)\leq p_+$ and $p$ is log-H\"older continuous, i.e. when there exists $c>0$ such that $|p(x)-p(y)|\leq -{c}/{\log(|x-y|)}$ for $|x-y|<{1}/{2}$. In the double-phase case $\vp_{dp}(x,s)=s^{p}+a(x)s^q$, $0\leq a\in C^{0,\alpha}(\Omega)$, $q/p\leq 1 +\alpha/n$, it covers non-uniformly elliptic operators satisfying
\[c_1^{(p,q)} |\xi|^{p}\leq \A(x,\xi)\cdot\xi\quad\text{and}\quad |\A(x,\xi)|\leq c_2^{(p,q)}\left(1+\gamma(x)+|\xi|^{p-1}+a(x)|\xi|^{q-1}\right)\]
with $0\leq\gamma\in L^{\wt\vp_{dp}(\cdot)}(\Omega).$ 
 Finally, in Orlicz case when $B\in C^1(\rp0)$ is a~doubling $N$-function it also retrieves typically considered conditions
\[c_1^B B(|\xi|)\leq \A(x,\xi)\cdot\xi\quad\text{and}\quad |\A(x,\xi)|\leq c_2^B\left(1+\gamma(x)+B'(|\xi|)\right),\quad\text{with  $0\leq \gamma\in L^{\wt B}(\Omega)$.}\] To give more examples one can consider problems in weighted Orlicz,  double phase with variable exponents, or multi-phase Orlicz cases, as long as $\vp(x,s)$ is comparable to a~function doubling  with respect to the second variable and satisfy nondegeneracy conditions (A0)--(A2).

\subsubsection*{\underline{Notation}} 
We give here only the notation necessary to understand the formulation of our main result, more preliminary information is presented in Section~\ref{sec:prelim}.\\ Distributional solutions to equation $-\Delta_p u=\mu$ when $p$ is small ($1<p<2-1/n$) do not necessarily belong to $W^{1,1}_{loc}(\Omega)$. The easiest example to give is the fundamental solution (when $\mu=\delta_0$). This restriction on the growth can be dispensed by the use of a~weaker derivative. We make use of the symmetric truncation $T_k:\R\to\R$ defined as 
\begin{equation}T_k(s)=\left\{\begin{array}{ll}s & |s|\leq k,\\
k\frac{s}{|s|}& |s|\geq k.
\end{array}\right. \label{Tk}
\end{equation}
Note that as a consequence of \cite[Lemma~2.1]{BBGGPV} for every function $u$, such that $T_t(u)\in W^{1,\vp(\cdot)}_0(\Omega)$ for every $t>0$ there exists a (unique) measurable function
$Z_u : \Omega \to \rn$ such that
\begin{equation}\label{gengrad} \nabla T_t(u) = \chi_{\{|u|<t\}} Z_u\qquad \hbox{ for
a.e. in $\Omega$ and for every $t > 0$.}
\end{equation}
With an abuse of~notation, we denote $Z_u$ simply by $\nabla u$ and call it a {\it generalized gradient}.
 
In order to introduce definitions of very weak solutions we define the space
\begin{equation}
\label{sp-tr}
\mathcal{ T}_0^{1,\vp(\cdot)}(\Omega)=\{u\text{ is measurable in }\Omega :\ T_t(u)\in W^{1,\vp(\cdot)}_0(\Omega)\text{ for every }t>0\},
\end{equation}
where $W^{1,\vp(\cdot)}_0(\Omega)$ is the completion of $C_0^\infty(\Omega)$ in norm of $W^{1,\vp(\cdot)}(\Omega)$. In fact, $u\in W_0^{1,\vp(\cdot)}(\Omega)$ if and only if $u\in\mathcal{T}_0^{1, \vp(\cdot)}(\Omega)$
and
$Z_u\in L^{\vp(\cdot)}(\Omega; \rn)$. In the latter case, $Z_u = \nabla u$ a.e. in $\Omega$.

\subsubsection*{\underline{Very weak solutions}}
We define two kinds of very weak solutions to problem~\eqref{eq:main} uder assumptions $(\A 1)$--$(\A 4)$ involving a  measure $\mu\in\Mb(\Omega)$. 
 

\medskip

Inspired by~\cite{BG,CiMa,DaMMa} we define solutions that can be reached in the limit of solutions to approximate problems.
\begin{definition}\label{def:sola}
A function $u\in 
\mathcal{ T}_0^{1,\vp(\cdot)}(\Omega)$ is called an {\em approximable solution} to problem~\eqref{eq:main} if $u$ is an a.e. limit of a sequence of solutions $\{u_s\}_s\subset  W_0^{1,\vp(\cdot)}(\Omega)$ to
\begin{flalign}\label{eq:sola-decomp}
\int_\Omega \A(x,\nabla u_s)\cdot\nabla \phi\, dx  =\int_\Omega \, \phi\,d\mu^s \quad\text{for any $\phi\in W^{1,\vp(\cdot)}_0(\Omega)\cap L^\infty(\Omega)$,}
\end{flalign} 
when $\{\mu^s\}\subset C^\infty(\Omega)$ is a sequence of bounded functions  that converges to $\mu$ weakly-$*$ in the space of measures  and such that \begin{equation}
\label{muslim}
\limsup_{s\to 0}|\mu^s|(\overline{B})\leq |\mu|(\overline{B})\quad\text{for every $B\subset\Omega$.}
\end{equation}
\end{definition} 
\noindent The definition seems {\em very weak} as we refrain from assuming any convergence of the gradients of approximate solutions. Nonetheless, this is enough to show in the proofs that for fixed $k$ also $\A(\cdot,\nabla (T_k u_s))\to \A(\cdot,\nabla(T_k u))$ a.e. in $\Omega$ and thus it is justified to call $u$ a~{\em solution} (though in a {\em very weak} sense).

\medskip

Having~\cite{DMMOP1} and Remark~\ref{rem:decomp} on measure decomposition (to parts being absolutely continuous and singular with respect to generalized capacity) we consider renormalized solutions according to the following definition.
\begin{definition}\label{def:rs}
A function $u\in \mathcal{ T}_0^{1,\vp(\cdot)}(\Omega)$ is called a {\em renormalized solution} to problem~\eqref{eq:main} with $\mu\in\Mb(\Omega)$, if 
\begin{itemize}
\item[(i)] for every $k>0$ one has $\ \A(x,\nabla (T_k u))\in L^{\wt\vp(\cdot)}(\Omega);$
\item[(ii)] $\mu$ is decomposed to $\mu=\mup+\musp-\musm$, with $\mup\in\MP(\Omega)$ and nonnegative $\musp ,\musm \in \big( \Mb(\Omega)\setminus\MP(\Omega)\big)\cup\{0\}$, then
\begin{flalign}
\nonumber \int_\Omega &\A(x,\nabla u)\cdot\nabla u\, h'(u)\phi\, dx+\int_\Omega \A(x,\nabla u)\cdot\nabla\phi\, h(u)\, dx\\ 
&=\int_\Omega  h(u)\phi\,d\mup(x)+h(+\infty)\int_\Omega \phi\,d\musp (x)-h(-\infty)\int_\Omega  \phi\,d\musm (x), \label{eq:renorm-decomp}
\end{flalign} 
holds for any $h\in W^{1,\infty}(\R)$ having $h'$ with compact support and for all $\phi\in C_0^{\infty}(\Omega)$, where $h(+\infty):=\lim_{r\to+\infty}h(r)$ and $h(-\infty):=\lim_{r\to+\infty}h(r)$ are well-defined as $h$ is constant close to infinities.
\end{itemize}
\end{definition} 
\noindent Recall that the assumption on the modular function $\vp$ are given in Section~\ref{sec:prelim}.

\medskip

\noindent Our main result reads as follows. 
 \begin{theorem}\label{theo:main} Let $\vp\in\Phi_c(\Omega)$ on a bounded Lipschitz domain $\Omega\subset\rn,$ $n\geq2$. Suppose that $\vp$ satisfies (aInc)$_p$, (aDec)$_q$, (A0), (A1), and (A2), whereas a vector field $\A:\Omega\times\rn\to\rn$ satisfies ($\A 1$)--($\A 4$). When $\mu\in\Mb(\Omega)$, then the following claims hold true.\begin{itemize}
 \item[ (i) ] There exists an approximable solution to problem~\eqref{eq:main}.  
 \item[ (ii) ] There exists a renormalized solution   to problem~\eqref{eq:main} satisfying~\eqref{eq:renorm-decomp} with measures such that ${\rm supp}\,\mup\subset\{|u|<\infty\},$ ${\rm supp}\,\musp\subset\cap_{k>0}\{u>k\}$, and ${\rm supp}\,\musm\subset\cap_{k>0}\{u<-k\}$.
 \item[(iii)] A function $u\in\mathcal{T}^{1,\vp(\cdot)}_0(\Omega)$ is an approximable solution from (i) if and only if it is a~renormalized solution from (ii).
 \item[(iv)] If additionally the measure datum is $\vp(\cdot)$-diffuse ($\mu\in\MP(\Omega)$), then approximable solution and renormalized solutions are unique.
\end{itemize} 
 \end{theorem}
\noindent As $h\equiv 1$ is an admissible choice in~\eqref{eq:renorm-decomp}, we get the following remark.
\begin{remark}  Under the assumptions of Theorem~\ref{theo:main} if $u$ is an approximable (equivalently, renormalized) solution, then
\[\int_\Omega \A(x,\nabla u)\cdot\nabla\phi\, dx =\int_\Omega \phi\,d\mu\qquad\text{for all }\ \phi\in C_0^\infty(\Omega),\]
so $u$ is then a solution in the distributional sense (which in particular is proven to exist).
 \end{remark}
 
 Moreover, for problems involving $\vp(\cdot)$-diffuse measures, by Theorem~\ref{theo:decomp} and Proposition~\ref{prop:sola-trunc-meas}, we can formulate  the following conclusion.
 \begin{corollary} Under the assumptions of Theorem~\ref{theo:main} if $u$ is an approximable (equivalently, renormalized) solution and $\mu\in\big(L^1(\Omega)+(W^{1,\vp(\cdot)}_0(\Omega))'\big)\cap\Mb(\Omega)$, then $u$ exists, is unique, and satisfies\[\limsup_{k\to\infty}\int_{\{k<|u|<k+1\}} \A(x,\nabla u)\cdot\nabla u\, dx=0.\] 
 \end{corollary}
 
  As a direct consequence of Theorem~\ref{theo:main} we retrieve the already classical existence results of~\cite{BGO,DMMOP1} involving $p$-Laplace operator, as well as variable exponent ones~\cite{Zhang,Zhang-Zhou}. We extend the existence results for problems in reflexive Orlicz spaces proven in~\cite{CiMa} towards inhomogeneity of the spaces, as well as we extend the uniqueness result from $L^1$ to the diffuse measure data. It should be noticed that renormalized solutions to general measure data problems with Orlicz growth were not studied so far. We also obtain the main goals of~\cite{pgisazg1,gw-e} within a different and a bit more restrictive functional framework (and slightly different kind of control on the modular function),  but allowing for essentially broader class of data and providing uniqueness.  To our best knowledge no results on equivalence of very weak solutions has been so far addressed in problems stated in generalized Orlicz spaces even in the $L^1$-data case, for the $p$-Laplace case we refer to~\cite{DMMOP1,kikutu}. Finding a setting where they essentially do not coincide would be interesting. 
 Given an interest one may expect developing our main goals further towards anisotropic or non-reflexive settings cf.~\cite{ACCZG,CGZG,pgisazg1}, as well as by involving lower-order terms in~\eqref{eq:main} as in~\cite{gw-e},  differential inclusions as in~\cite{dkg}, or systems of equations.
 
 There is some available information on the regularity of our very weak solutions following from comparison to solutions to problems with Orlicz growth. The conditions on $\vp(\cdot)$ imply that there exists a Young function $B:\rp0\to\rp0$ such that $B(s)\leq \vp(x,s)$ for a.a. $x\in\Omega$ and all $s\geq 0$. Then  any of the very weak solutions of Theorem~\ref{theo:main}  belongs to $\mathcal{T}_0^{1,\vp(\cdot)}(\Omega)\subset \mathcal{T}_0^{1,B}(\Omega)$. Thus, we can get the same regularity of these solutions and their gradients expressed in Orlicz-Marcinkiewicz scale as in~\cite[Theorem 3.2]{CiMa}. See~\cite[Example 3.4]{CiMa} for applications with particular growth of~$B$ (including Zygmund-type ones). On the other hand, precise informations on the local behaviour of solutions to problems with Orlicz growth obtained as a consequense of Wolff-potential estimates can be found in~\cite{CGZG-Wolff} depending on the scale of datum (in Orlicz versions of Lorentz, Marcinkiewicz, and Morrey scales). When the growth of $B$ is super-quadratic, we have also provided more precise information of the gradient of solutions. In fact, \cite{gradest} gives the Orlicz-Lorentz-Morrey-type regularity for gradients of solutions to problems involving related classes of measures, moreover,~\cite{gradest2} describes the regularizing effect of the lower-order term (in the same scale). For Riesz potential estimates for such problems see~\cite{Baroni}.

 The main ideas of the proofs follows many seminal papers including~\cite{BGO,BBGGPV,BM,DMMOP1}  and involve analysis of fine convergence of solutions of some approximate problems. Nonetheless, the functional setting is far more demanding. In fact, we employ a lot of very recent results on structural properties of the generalized Orlicz spaces and  nonstandard capacities, see e.g.~\cite{bahaha,demi,hahabook,haju}, and study properties of measures exposing certain capacitary properties.

As for organization -- after preliminary part, the measure characterization is proven in Section~\ref{sec:meas}, Section~\ref{sec:appr} is devoted to approximate problems. 
Approximable solutions are investigated in Section~\ref{sec:solas}, while renormalized oned in Section~\ref{sec:rs}. Uniqueness is proven in Section~\ref{sec:uniq}. The summary of the main proof is presented in Section~\ref{sec:main-proof}.

\section{Preliminaries} \label{sec:prelim}

\subsection{Notation} By $\Omega$ we always mean a bounded set in $\rn$, $n\geq 2$. We shall make use of symmetric truncations of a real-valued function\begin{equation}
\label{trunc} T_k (s)=\max\{-k,\min\{s,k\}\}.
\end{equation}
Also, we make use of a Lipschitz continuous cut-off function $\psi_l:\R\to\R$ by 
\begin{equation}\label{psil}
\psi_l(r):= 
\min\{(l+1-|r|)^+,1\}.
\end{equation}

By $\mu_1\ll\mu_2$ we denote we mean that $\mu_1$ is absolutely continuous with respect to $\mu_2$.

We study spaces of functions defined in $\Omega$, $\R$, or $\rn$.  $L^0(\Omega)$ denotes the set of measurable functions defined on $\Omega$,    $C_0(\Omega)$ are continuous functions taking value zero on $\partial\Omega$, while  $C_b(\Omega)$ -- continuous  functions bounded on $\Omega$; $\Mb(\Omega)$ are Radon measures with bounded total variation in~$\Omega$; $\MP(\Omega)$ -- bounded Radon measures diffuse with respect to $\vp(\cdot)$-capacity. If $\mu\in\Mb(\Omega)$, $E$ is a Borel set included in $\Omega$, the measure $\mu\mres E$ is defined by $(\mu\mres E)(B)=\mu(E\cap B)$ for any Borel set $B\subset\Omega$. If $\mu\in\Mb(\Omega)$ is such that $\mu=\mu\mres E,$ then we say that $\mu$ is concentrated on $E$. In general, one cannot define the smallest set (in the sense of inclusion) where the measure is concentrated. By $L^1(\Omega,\mu)$ we denote classically functions with absolute value integrable with respect to $\mu$, shortened to $L^1(\Omega)$ if $\mu$ is Lebesgue's measure.\\ When $\mu_k,\mu\in\Mb(\Omega)$, we say that $\mu_k\to\mu$ weakly-$\ast$ in the space of measures if
\[\lim_{k\to\infty}\int_\Omega \phi\, d\mu_k=\int_\Omega \phi\, d\mu\qquad\text{for every }\phi\in C_0(\Omega).\]


\begin{lemma} \label{lem:ae}
If $g_n:\Omega\to \R$ are measurable functions converging 
to $g$ almost everywhere, then for each regular value $t$ 
of the limit function $g$ we have $\mathds{1}_{\{t<|g_n|\}}\xrightarrow[n\to\infty]{}\mathds{1}_{\{t<|g|\}}$ a.e. in $\Omega$. // Here the term `regular value' denotes a value $t$ such that $g^{-1}(t)$ has measure zero.
\end{lemma} 
\begin{lemma} \label{lem:TM1}
 Suppose $w_n\to w$ in $L^1(\Omega)$, $v_n,v\in L^\infty(\Omega)$, and $v_n\to v$ a.e. in $\Omega$. Then $w_n v_n \to  w v$ in $L^1(\Omega)$.
\end{lemma}

\subsection{Generalized Orlicz functions}
Essentially the fairly general framework we employ comes from the monograph~\cite{hahabook}. For other recent developments within the closely related functional settings see also~\cite{yags,bahaha,CGWKSG,CUH,haju}.

A real-valued function is $L$-almost increasing, $L\geq 1$, if $Lf(s) \geq f(t)$ for $s > t$. $L$-almost decreasing is defined analogously.

\begin{definition}  We say that $\vp:\Omega\times\rp0\to[0,\infty]$ is a convex $\Phi$--function, and write $\vp\in\Phi_c(\Omega)$, if the following conditions
hold:
\begin{itemize}
\item[(i)]  For every $s\in\rp0$ the function $x\mapsto\vp(x, t)$ is
measurable and for a.e. $x\in\Omega$ the function $s\mapsto\vp(x, t)$ is increasing, convex, and left-continuous.
\item[(ii)] $\vp(x, 0) = \lim_{s\to 0^+} \vp(x, s) = 0$ and $\lim_{s\to \infty} \vp(x, s) =
\infty$ for a.e. $x\in\Omega$.
\end{itemize}
\end{definition}
\noindent Further, we say that $\vp\in\Phi_c(\Omega)$ satisfies\begin{itemize}
\item[(aInc)$_p$] if there exists $L\geq 1$ such that $s\mapsto\vp(x, s)/s^p$ is $L$-almost increasing in $\rp0$ for every $x\in\Omega$, 
\item[(aDec)$_q$] if there exists $L\geq 1$ such that $s\mapsto\vp(x, s)/s^q$ is $L$-almost decreasing in $\rp0$ for every $x\in\Omega$.
\end{itemize}
We write (aInc), if there exist $p > 1$ such that (aInc)$_p$ holds and (aDec) if there exist $q > 1$ such that (aDec)$_q$ holds. The corresponding conditions with $L = 1$ are denoted by (Inc) or (Dec). \\
We shall consider those $\vp\in\Phi_c(\Omega)$, which satisfy the following set of conditions.
\begin{itemize}
\item[(A0)] There exists $\beta_0\in (0, 1]$ such that $\vp(x, \beta_0) \leq 1$ and
$\vp(x, 1/\beta_0) \geq 1$ for all $x\in\Omega$.
\item[(A1)] There exists $\beta_1\in(0,1)$, such that for every ball $B$ with $|B|\leq 1$ it holds that \[\beta_1\vp^{-1}(x,s)\leq\vp^{-1}(y,s)\quad\text{for every $s\in [1,1/|B|]$ and a.e. $x,y\in B\cap\Omega$}.\]
\item[(A2)] For every $s$ there exists $\beta_2\in(0,1]$ and $h\in L^1(\Omega)\cap L^\infty(\Omega)$, such that  \[\vp(x,\beta_2 r)\leq\vp(y,r) +h(x)+h(y)\quad\text{for a.e. $x,y\in \Omega$ whenever $\vp(y,r)\in[0,s]$}.\]
\end{itemize}
Condition (A0) is imposed in order to exclude degeneracy, while (A1) can be interpreted as local continuity. Fundamental role is played also by (A2) which imposes balance of the growth as it is equivalent to existence $\vp_\infty\in\Phi_c$, $h\in L^1(\Omega)\cap L^\infty(\Omega)$, $\beta_2'\in(0,1]$, $s>0$, such that 
for a.e. $x\in \Omega$, $\vp_\infty(r)\in[0,s]$, and $\vp(y,r)\in[0,s]$ it holds that
\[\vp(x,\beta_2' r)\leq\vp_\infty(r) +h(x)\quad\text{and}\quad \vp_\infty(\beta_2' r)\leq\vp(x,r) +h(x).\]
 
We say that a function $\vp$ satisfies $\Delta_2$-condition (and write $\vp\in\Delta_2$) if there exists a~constant $c>0$, such that for every $s\geq 0$ it holds $\vp(x,2s)\leq c(\vp(x,s)+1)$.  When a function $\vp\in \Phi_c(\Omega)$ satisfies {\rm (aInc)$_p$} and {\rm (aDec)$_q$}, then it is comparable to some $\psi\in\Delta_2$.

\emph{The Young conjugate} of $\vp\in\Phi_c(\Omega)$  is the function $\wt\vp:\Omega\times\rp0\to[0,\infty]$
defined as
$$\wt \vp (x,s) = \sup\{r \cdot s - \vp(x,r):\ r \in \rp0\}.
$$
Note that Young conjugation is involute, i.e. $\wt{(\wt \vp )}=\vp$. Moreover, if $\vp\in\Phi_c(\Omega)$, then $\wt\vp\in\Phi_c(\Omega)$. 
If $\wt\vp\in\Delta_2,$ we say that $\vp$ satisfies $\nabla_2$-condition and denote it by $\vp\in\nabla_2$. If  $\vp,\wt\vp\in\Delta_2$, then we call $\vp$ a doubling function. If $\vp\in \Phi_c(\Omega)$ satisfies {\rm (aInc)$_p$} and {\rm (aDec)$_q$}, $\wt\vp$ is comparable to some $\wt\psi\in\Delta_2$, so we can assume that functions within our framework are doubling.

For $\vp\in\Phi_c(\Omega)$, the following inequality of Fenchel--Young type holds true
\begin{equation}
\label{in:FY}
rs\leq \vp(x,r)+\wt\vp(x,s).
\end{equation}
In fact, within our framework with, since $\vp$ is comparable to a doubling function there exist some constants depending only on $\vp$ for which we have
\begin{equation}
\label{doubl-star}\wt\vp\left(x, {\vp(x,s)}/{s}\right)\sim \vp(x,s) \quad\text{for a.e. }\ x\in\Omega\ \text{ and all }\ s>0.
\end{equation} 

\subsection{Function spaces}\label{ssec:spaces} We always deal with spaces generated by $\vp\in\Phi_c(\Omega)$ satisfying (aInc)$_p$, (aDec)$_q$, (A0), (A1), and (A2). For $f\in L^0(\Omega)$ we define {\em the modular} $\vr_{\vp(\cdot),\Omega}$   by
\[\vr_{\vp(\cdot),\Omega} (f)=\int_\Omega\vp (x, | f(x)|) dx.\]
When it is clear from the context we omit assigning the domain.

\noindent Musielak--Orlicz
space is defined as the set
\[L^{\vp(\cdot)} (\Omega)= \{f \in L^0(\Omega): \lim_{\lambda\to 0^+}\vr_\vp(\lambda f) = 0\}\]
endowed with the Luxemburg norm
\[\|f\|_{\vp(\cdot)}=\inf \{\lambda > 0 : \vr_{\vp(\cdot)} \left(\tfrac 1\lambda f\right) \leq 1\} .\]
For $\vp\in\Phi_c(\Omega)$, the space $L^{\vp(\cdot)}(\Omega)$ is a Banach space~\cite[Theorem~2.3.13]{hahabook}. Moreover, the following H\"older inequality holds true\begin{equation}
\label{in:Hold}\|fg\|_{L^1(\Omega)}\leq 2\|f\|_{L^{\vp(\cdot)}(\Omega)}\|g\|_{L^{\wt\vp(\cdot)}(\Omega)}.
\end{equation}
Sometimes it would be convenient for us to denote vector-valued functions integrable with the modular as $(L^{\vp(\cdot)}(\Omega))^n$. Since there is no difference between claiming $H=(H_1,\dots,H_n)\in (L^{\vp(\cdot)}(\Omega))^n$ and $|H|\in L^{\vp(\cdot)}(\Omega)$, we are not very careful with stressing it in the sequel. A function $f\in L^{\vp(\cdot)} (\Omega)$ belongs to {\em Musielak-Orlicz-Sobolev space} $W^{1,\vp(\cdot)} (\Omega)$, if its distributional partial derivatives $\partial_1 f,\dots,\partial_n f$ exist and belong to $L^{\vp(\cdot)} (\Omega)$ too. Because of the growth conditions $W^{1,\vp(\cdot)}(\Omega)$ is a separable and reflexive space. Moreover, smooth functions are dense there. As a zero-trace space $W_0^{1,\vp(\cdot)}(\Omega)$ we mean the closure of $C_0^\infty(\Omega)$ in $W^{1,\vp(\cdot)}(\Omega)$. In fact, due to~\cite[Theorem~6.2.8]{hahabook} given a~bounded domain $\Omega$ there exists a~constant $c=c(n,\Omega)>0,$ such that for any $u\in W_0^{1,\vp(\cdot)}(\Omega)$ it holds that
\begin{equation}
\|u\|_{L^{\vp(\cdot)}(\Omega)}\leq c \|\nabla u\|_{L^{\vp(\cdot)}(\Omega)}.\label{in:Sob}
\end{equation}
Moreover, \cite[Theorem 6.3.7]{hahabook} yields that\begin{equation}
\label{emb} W^{1,\vp(\cdot)}_0(\Omega)\hookrightarrow\hookrightarrow L^{\vp(\cdot)}(\Omega),
\end{equation}
where `$\hookrightarrow\hookrightarrow$' stands for a compact embedding.
   

\begin{remark}\cite{hahabook}\label{rem:top} If  $\vp\in \Phi_c(\Omega)$ satisfies {\rm (aInc)$_p$}, {\rm (aDec)$_q$}, (A0), (A1), (A2), strong (norm) topology of $W^{1,\vp(\cdot)}(\Omega)$ coincides with the sequensional modular topology. Moreover, smooth functions are dense in this space in both topologies. Thus $W_0^{1,\vp(\cdot)}(\Omega)$, under our assumptions, is a closure of $C_0^\infty(\Omega)$ with respect to the modular topology of gradients in $L^{\vp(\cdot)}(\Omega)$.
\end{remark}
 
 
\noindent Space $(W^{1,\vp(\cdot)}_0(\Omega))'$ is considered endowed with the norm
\[\| H\|_{(W^{1,\vp(\cdot)}_0(\Omega))'} =\sup\left\{\frac{\langle H, v\rangle}{\| v\|_{W^{1,\vp(\cdot)}(\Omega)}}:\quad  v\in W^{1,\vp(\cdot)}_0(\Omega)\right\}.
\]

\subsection{The operator} Let us motivate that the growth and coercivity conditions from $(\A 1)$-$(\A 4)$ imply the expected proper definition of the operator involved in problem~\eqref{eq:main}. We consider the operator $\ \oppA : W^{1,\vp(\cdot)}_0(\Omega) \to (W^{1,\vp(\cdot)}_0(\Omega))'\ $ defined as
\[\oppA(v)=-\dv\A(x,\nabla v),\]
that is acting 
\begin{flalign}\label{oppA}
\langle\oppA( v),w\rangle:=\int_{\Omega}\A(x,\nabla v)\cdot \nabla w  \,dx\quad \text{for}\quad w\in C^{\infty}_{0}(\Omega),
\end{flalign}
where $\langle \cdot, \cdot \rangle$ denotes dual pairing between reflexive Banach spaces $W^{1,\vp(\cdot)}(\Omega))$ and $(W^{1,\vp(\cdot)}(\Omega))'$ is well-defined. Note that when $v\in W^{1,\vp(\cdot)}(\Omega)$ and $w\in C_0^\infty(\Omega)$, growth condition $(\A 2)$, H\"older's inequality~\eqref{in:Hold}, equivalence~\eqref{doubl-star}  justify that
\begin{flalign}
\nonumber\snr{\langle \oppA (v),w \rangle}\le &\, c\int_{\Omega}\frac{\vp(x,\snr{\nabla v})}{\snr{\nabla v}}\snr{\nabla w} \ dx \le c\left \| \frac{\vp(\cdot,\snr{\nabla v})}{\snr{\nabla v}}\right \|_{L^{\wt \vp(\cdot)}(\Omega)}\nr{\nabla w}_{L^{\vp(\cdot)}(\Omega)}\nonumber \\
\le &\, c\nr{\nabla v}_{L^{ \vp(\cdot)}(\Omega)}\nr{\nabla w}_{L^{\vp(\cdot)}(\Omega)}\le c\nr{w}_{W^{1,\vp(\cdot)}(\Omega)}.\label{op}
\end{flalign}
By density argument, the operator is well-defined on  $W^{1,\vp(\cdot)}_0(\Omega)$.

\medskip

What is more,  by ($\A$1)--($\A$2) and \cite[Lemma~4.12]{CZG-gOp} we have the following.

\begin{remark} For $u\in\mathcal{T}^{1,\vp(\cdot)}(\Omega)$, such that for some $M,k_0>0$ it holds that $\vr_{\vp(\cdot),\Omega}(\nabla T_k  u)\leq Mk$ for all $k>k_0$, there exists a continuous function $\zeta:[0,|\Omega|]\to\rp0$, such that $\lim_{s\to 0^+}\zeta(s)=0$ and for any measurable $E\subset\Omega$\[\int_E |\A(x,\nabla u)|\,dx\leq\zeta(|E|),\]
 where `$\nabla$' is understood as in~\eqref{gengrad}. In particular, $\A(\cdot,\nabla u)\in L^1(\Omega)$.\label{rem:A-int}
\end{remark}

\subsection{Capacities}\label{ssec:cap} Understanding capacities is needed to describe pointwise behavior of Sobolev functions. We employ the generalization of classical notions of capacities, cf.~\cite{AH,hekima,Re}, as well as unconventional ones~\cite{ks,MSZ} to the Musielak-Orlicz-Sobolev setting according to~\cite{bahaha,haju}. 

For a set $E\subset\rn$ we define
\[ {S}_{1,\vp(\cdot)}(E):=\{0\leq \phi\in W^{1,\vp(\cdot)}(\rn):\ \phi\geq 1\ \text{in an open set containing } E\}\]
and its generalized Orlicz capacity of Sobolev type (called later {\em $W^{1,\vp(\cdot)}$--capacity}) by \[\capP (E)=\inf_{\phi\in  {S}_{1,\vp(\cdot)}(E)}\left\{
\int_\rn \vp(x, \phi)+\vp(x,|\nabla \phi|)\,dx\right\}
.\]



We shall consider {\em generalized relative $\vp(\cdot)$-capacity} $\CapP$. With this aim for every $K$ compact in $\Omega\subset\rn$ let us denote\begin{equation}
\label{def:R}
\mathcal{R}_{\vp(\cdot)}(K,\Omega) := \{v\in W^{1,\vp(\cdot)}(\Omega)\cap C_0(\Omega):\quad v\geq 1\ \text{ on $K$ and }\ v \geq 0\}
\end{equation}
and set
\[\CapP(K,\Omega) :=\inf\left\{ \vr_{\vp(\cdot),K}(\phi):\ \ \phi\in \mathcal{R}_{\vp(\cdot)}(K,\Omega )\right\} .\]
For open  sets $A\subset\Omega$ we define\[\CapP (A,\Omega)=\sup\left\{\CapP(K,\Omega):\quad K\subset A \text{ and } K \text{ is compact in }A\right\} \]
and finally, if $E\subset\Omega$ is an arbitrary set \[\CapP (E,\Omega)=\inf\left\{\CapP(A,\Omega):\quad E\subset A \text{ and } A \text{ is open in }\Omega\right\}. \]
This notion of capacity enjoys all fundamental properties of classical capacities~\cite{bahaha,haju}.

Let us pay some attention to sets of zero capacity. If $B_R$ is a ball in $\rn$, $E\subset B_R$ and $\CapP(E,B_R)=0$, then $|E|=0$. Having bounded $\Omega\subset\rn$ for a set $E\subset\Omega$ we have $\CapP(E,\Omega)=0$ if and only if $\capP(E)=0$. What is more, each set of $W^{1,\vp(\cdot)}$-capacity zero is contained in a Borel set of $W^{1,\vp(\cdot)}$-capacity zero. Countable union of sets of $W^{1,\vp(\cdot)}$-capacity zero has $W^{1,\vp(\cdot)}$--capacity zero.

Function $u$ is called {\em $\capP$--quasicontinuous} if for every $\e>0$ there exists an open set $U$ with $\capP (U)<\e$, such that $f$ restricted to $\Omega\setminus U$ is continuous. We say that a~claim holds {\em $\vp(\cdot)$--quasieverywhere} if it holds outside a set of Sobolev $\vp(\cdot)$--capacity zero.  A set $E\subset \Omega$ is said to be {\em $\capP$-quasiopen} if for every $\ve>0$ there exists a decreasing an open set $U$ such that $E\subset U\subset \Omega$ and $\capP(U\setminus E)\leq\ve$.

\begin{lemma}\label{lem:almost-uniform} For every Cauchy
sequence in $W^{1,\vp(\cdot)} (\Omega)$ (equivalently under our regime, with respect to the $W^{1,\vp(\cdot)} (\Omega)$--modular topology) of functions from $C(\rn)\cap W^{1,\vp(\cdot)} (\Omega)$ there is a subsequence which converges pointwise $\capP$-quasieverywhere in $\Omega$. Moreover, the convergence is uniform outside a set of arbitrary small capacity $\capP$.
\end{lemma}

In the sequel we shall always identify $u$ with its $\capP$-quasicontinuous representative. 


\begin{lemma}\label{lem:sola-qc}For each $u\in\mathcal{T}^{1,\vp(\cdot)}_0(\Omega)$  there exists a unique $\capP$--quasicontinuous function $v \in \mathcal{T}^{1,\vp(\cdot)}_0(\Omega)$ such that $u=v$ almost everywhere in $\Omega$.
\end{lemma}  


As a direct consequence of Lemma~\ref{lem:sola-qc}, we have the following observations.

\begin{lemma}\label{lem:qc-open} For a $\capP$-quasicontinuous function $u
$ and $k>0$, the sets $\{|u|>k\}$ and $\{|u|<k\}$ are $\capP$-quasi open.\end{lemma}
\begin{lemma}\label{lem:exh}
For every  $\capP$-quasiopen set $U\subset\Omega$ there exists an increasing sequence $\{v_n\}$ of nonnegative functions in $W^{1,\vp(\cdot)}_0(\Omega)$ which converges to $\mathds{1}_U$ $\capP$-quasi everywhere in $\Omega$.
\end{lemma}

\begin{lemma}\label{lem:summability}
If $\mup\in\MP(\Omega)$ and $u\in W^{1,\vp(\cdot)}_0(\Omega)$, then $\capP$-quasicontinuous representative $\widehat{u}$ of $u$ is measurable with respect to $\mup$. If additionally $u\in L^\infty$, then $\widehat{u}\in L^\infty(\Omega,\mup)\subset L^1(\Omega,\mup).$
\end{lemma}

\section{Measure characterization}\label{sec:meas}

In order to prove Theorem~\ref{theo:decomp} let us concentrate on the continuity of $\mu\in \big(L^1(\Omega)+(W^{1,{\vp(\cdot)}}_0(\Omega))'\big)\cap \Mb(\Omega)$  with respect to  the generalized capacity. Notice that for a nonnegative measure having decomposition $\mu =f-\dv G \in \big(L^1(\Omega)+(W^{1,{\vp(\cdot)}}_0(\Omega))'\big)\cap \Mb(\Omega)$ with $f\in L^1(\Omega)$ and $G\in (L^{\wt\vp(\cdot)}(\Omega))^n$ and for an arbitrary set $E\subset\Omega$, we have
\[\mu(E)\leq \int_E f\,\phi\,dx+\int_E G\cdot\nabla\phi\,dx\leq \|f\|_{L^1(E)}\|\phi\|_{L^\infty(R)}+\|G\|_{(L^{\wt\vp(\cdot)}(\Omega))^n}\|\nabla \phi\|_{(L^{{\vp(\cdot)}}(\Omega))^n}\]
for every ${\phi\in W^{1,{\vp(\cdot)}}(\Omega)}.$  In general, it is possible that a set has zero measure, but positive capacity. This is excluded if the measure enjoys the above decomposition.
\begin{lemma}\label{lem:cap0}
If $\mu\in \big(L^1(\Omega)+(W^{1,{\vp(\cdot)}}_0(\Omega))'\big)\cap \Mb(\Omega)$ and a set $E\subset\Omega$ is such that $\CapP(E,\Omega)=0$, then $\mu(E)=0$.
\end{lemma}
\begin{proof}By the assumption there exist $f\in L^1(\Omega)$ and $G\in (L^{\wt\vp(\cdot)}(\Omega))^n$, such that $\mu=f-\dv G$ in the sense of distributions. Then obviously $\mu\in\Mb(\Omega)$. Moreover, there exists a Borel set $E_0\supset E$ with $\CapP(E_B,\Omega)=0$. We fix compact $K\subset E_B$ and open $\Omega'\subset\Omega$, such that $K\subset\Omega'$. Let us consider a sequence $\{\phi_j\}_j\subset C_0^\infty(\Omega')$ of functions such that \[\text{$K\subset\{\phi_j= 1\},\ $ $0\leq\phi_j\leq 1\ $ and }\int_{\Omega'}\vp(x,\tfrac{1}{\lambda}| \nabla\phi_j|)\,dx\xrightarrow[j\to\infty]{}0\ \text{ for every $\lambda>0$},\]then
\[|\mu(K)|\leq \left|\int_{\Omega'}\phi_j\,d\mu\right|=\left|\int_{\Omega'}f\,\phi_j\,dx + \int_{\Omega'}G\cdot\nabla \phi_j\,dx\right|.\]
Consider $\{f_\delta\}_\delta$ of smooth functions converging to $f$ strongly in $L^1(\Omega)$ and such that $f_\delta\leq 2 f$ a.e. in $\Omega$. For fixed $\delta$, due to H\"older inequality~\eqref{in:Hold} and then the Sobolev one~\eqref{in:Sob}, we have
\begin{flalign*}
|\mu(K)|&\leq \int_{\Omega'}|f_\delta-f|\,|\phi_j|\,dx+\int_{\Omega'}|f_\delta|\,|\phi_j|\,dx + \int_{\Omega'}|G\cdot\nabla \phi_j|\,dx\\
&\leq \|f_\delta-f\|_{L^1(\Omega')}\,\|\phi_j\|_{L^\infty(\Omega')}+2\|f_\delta\|_{L^{\wt\vp(\cdot)}(\Omega')}\,\|\phi_j\|_{L^{\vp(\cdot)}(\Omega')} + \|G\|_{L^{\wt\vp(\cdot)}(\Omega')}\,\|\nabla \phi_j\|_{L^{\vp(\cdot)}(\Omega')} \\
&\leq \|f_\delta-f\|_{L^1(\Omega')} +\left(c\|f_\delta\|_{L^{\wt\vp(\cdot)}(\Omega')}  + \|G\|_{L^{\wt\vp(\cdot)}(\Omega')}\right)\|\nabla \phi_j\|_{L^{\vp(\cdot)}(\Omega')} .
\end{flalign*}
As the norm convergence is equivalent to modular one, we are able to choose $j$ large enough to estimate $\|\nabla \phi_j\|_{L^{\vp(\cdot)}(\Omega')}\leq \|f_\delta-f\|_{L^1(\Omega')} $. Then we take infimum with respect to $\delta$ and get
\[\mu(E)\leq\mu(E_B)=\sup\{\mu(K):\ K\subset E_B,\ K \text{ compact}\}=0,\] 
which ends the proof.\end{proof}

We are in position to prove Theorem~\ref{theo:decomp}. We take basic ideas from~\cite{BGO} with classical growth. Similar reasoning in variable exponent setting is given in~\cite{Zhang}.  

\begin{proof}[Proof of Theorem~\ref{theo:decomp}]

The implication: if $\mu$ belongs to $\big( L^1(\Omega)+(W_0^{1,{\vp(\cdot)}}(\Omega))'\big)\cap \Mb(\Omega)$, then $\mu\in \mathcal{M}^{\vp(\cdot)}_b(\Omega)$ is provided in Lemma~\ref{lem:cap0}. We shall concentrate now on the essentially more demanding converse, that is, if $\mu_{\vp(\cdot)}\in \mathcal{M}^{\vp(\cdot)}_b(\Omega)$, then $\mu_{\vp(\cdot)}\in L^1(\Omega) + (W^{1,{\vp(\cdot)}}_0(\Omega))'$.

\medskip

\noindent {\em Step 1. Initial decomposition. }\\ The aim of this step is to show that for a nonnegative $\mup\in \MP(\Omega)$ we can find a nonnegative measure $\gam\in(W_0^{1,{\vp(\cdot)}}(\Omega))'$  and nonnegative  Borel measurable function $h\in L^1(\Omega,\gam)$ such that $d\mup=h\,d\gam$. 

For any $\wt u\in W^{1,{\vp(\cdot)}}(\Omega)$ we can find its uniquely (a.e.) defined $\capP$-quasi-continuous representative denoted by $u$ (see Lemma~\ref{lem:sola-qc}). We define a functional $\mathcal{F}: W^{1,{\vp(\cdot)}}_0(\Omega)\to [0,\infty]$ by
\[\mathcal{F}[u]=\int_\Omega  u_+\,d\mup\]
and observe that it is convex and lower semicontinuous on a separable space $W^{1,{\vp(\cdot)}}_0(\Omega)$. Thus, $\mathcal{F}$ can be expressed as a supremum of a countable family of continuous affine functions. In fact, there exist sequences of functions $\{\xi_n\}_n\subset (W^{1,{\vp(\cdot)}}_0(\Omega))'$ and numbers $\{a_n\}_n\subset\rn$ such that
\[\mathcal{F}[u]=\sup_{n\in\N}\,\{\langle\xi_n,u\rangle-a_n\}\quad\text{for all }\ u\in W^{1,{\vp(\cdot)}}_0(\Omega).\]
Then, for any $s>0$, $s\mathcal{F}[u]=\mathcal{F}[su]\geq s\langle\xi_n,u\rangle-a_n$ for every $n$. By dividing by $s$ and letting $s\to\infty$ we get $\mathcal{F}[u]\geq \langle\xi_n,u\rangle$ for all $u\in W_0^{1,{\vp(\cdot)}}(\Omega)$. Since $\mathcal{F}[0]= 0$, it follows that $a_n\geq 0$. Thus $\mathcal{F}[u]\geq \sup_{n\in\N} \langle\xi_n,u\rangle\geq \sup_{n\in\N}\{\langle\xi_n,u\rangle-a_n\}=\mathcal{F}[u]$
 and, consequently, \begin{equation}
 \label{F} 
\mathcal{F}[u]=\sup_{n\in\N}\, \langle\xi_n,u\rangle.
 \end{equation} In turn, for all $\phi\in C_0^\infty(\Omega)$ we have
 \[ \langle\xi_n,\phi\rangle\leq  \sup_{n\in\N}\,\langle\xi_n,\phi\rangle=\mathcal{F}[\phi]=\int_\Omega \phi_+\,d\mup\leq \|\mup\|_{\Mb(\Omega)} \|\phi\|_{L^\infty(\Omega)}.\]
 By the same arguments for $-\vp$ we get \[| \langle\xi_n,\phi\rangle|\leq  \|\mup\|_{\Mb(\Omega)} \|\phi\|_{L^\infty(\Omega)}\]
 implying that $\xi_n\in (W^{1,{\vp(\cdot)}}_0 (\Omega))'\cap \Mb(\Omega).$ By the Riesz representation theorem there exists nonnegative $\xinm\in\Mb(\Omega)$, such that
 \[\langle\xi_n,\phi\rangle=\int_\Omega \phi \ d\xinm \quad\text{for all }\phi\in C_0^\infty(\Omega).\] 
We observe that \begin{equation}
 \label{xinm-mu}
\xinm\leq\mup\quad\text{and}\quad\|\xinm\|_{\Mb(\Omega)}\leq \|\mup\|_{\Mb(\Omega)}.
 \end{equation}
 Let us define\begin{equation}
 \label{eta}\eta=\sum_{n=1}^\infty\frac{\xi_n}{2^n(\|\xi_n\|_{(W^{1,{\vp(\cdot)}}_0(\Omega))'}+1)}
 \end{equation}
 and note that the series in absolutely convergent in $(W^{1,{\vp(\cdot)}}_0(\Omega))'$. Hence, whenever $\phi\in C_0^\infty(\Omega)$ we can write
 \begin{flalign*}
|\langle\eta,\phi\rangle|&\leq \sum_{n=1}^\infty\frac{|\langle \xi_n,\phi\rangle|}{2^n(\|\xi_n\|_{(W^{1,{\vp(\cdot)}}_0(\Omega))'}+1)}\\
&\leq \sum_{n=1}^\infty\frac{\|\xinm\|_{\Mb(\Omega)}}{2^n}\|\phi\|_{L^\infty(\Omega)} \leq \|\mup\|_{\Mb(\Omega)}\|\phi\|_{L^\infty(\Omega)}
\end{flalign*}
and $\eta\in (W^{1,{\vp(\cdot)}}_0(\Omega))'\cap \Mb(\Omega)$ too. Taking now\[\etam=\sum_{n=1}^\infty\frac{\xinm}{2^n(\|\xi_n\|_{(W^{1,\vp(\cdot)}_0(\Omega))'}+1)}\]
we deal with the series of positive elements that is absolutely convergent in $ \Mb(\Omega)$. Moreover, $\xinm\ll \etam$ and thus for every $n$ there exists a nonnegative function $h_n\in L^1(\Omega,d\etam)$ such that $d\xinm=h_n\, d\etam$ and -- according to~\eqref{F} -- we get that
\begin{equation}
\label{functional}\langle\mup,\phi\rangle=
\int_\Omega \phi\,d\mup=\sup_{n\in\N}
\int_\Omega \phi\,d\xinm=\sup_{n\in\N}\int_\Omega h_n\,\phi\,d\etam\qquad\text{for any }\ \phi\in C_0^\infty(\Omega).
\end{equation} 
On the other hand,~\eqref{xinm-mu} ensures that $h_n\etam\leq\mup$, i.e. for any measurable set $E\subset\Omega$ and every $n$ we have \begin{equation}
\int_E h_n\,d \etam\leq \mup(E).
\end{equation} 
We denote $h_{\max}^{k}=\max\{h_1(x),\dots,h_k(x)\}$ and
\begin{equation}
\label{Es}E^{j,k}=\{x\in E:\ \ h_{\max}^{j}(x)>h_i(x)\ \text{ for every }\ i=1,\dots,k-1\}.
\end{equation}
Then $E^{j,k}$ for $i=1,\dots,k$ are pairwise disjoint and $E=\cup_{j=1}^k E^{j,k}$, so\begin{equation}
\int_E h_{\max}^k(x)\,d \etam\leq \sum_{j=1}^k
\int_{E^{j,k}} h_{\max}^k(x)\,d \etam\leq  \sum_{j=1}^k \mup(E^{j,k})=\mup(E).
\end{equation} 
Letting $k\to\infty$ and taking $h(x)=\sup_{l\in\N} h_l(x)$, we get for any measurable set $E\subset\Omega$\begin{equation}
\int_E h\,d \etam\leq \mup(E).
\end{equation} 
Having~\eqref{functional}, for every nonnegative $\phi\in C_0^\infty$ we have
\begin{flalign*}
\int_\Omega \phi\,d\mup=\sup_{l\in\N} \int_\Omega h_l\, \phi\,d\etam\leq \int_\Omega h \,\phi\,d\etam\leq \int_\Omega \phi\,d\mup
\end{flalign*}
that is \[d\mup=h\,d\etam.\] Since $\mup(\Omega)\in\Mb(\Omega)$, we deduce that $h\in L^1(\Omega,d\etam)$ and the aim of this step is achieved with $\gam=\etam\in(W^{1,{\vp(\cdot)}}_0(\Omega))'$.

\medskip

\noindent {\em Step 2. Auxiliary sequence of measures. }\\
 Let $\{K_i\}_i$ be an increasing sequence of sets compact  in $\Omega$, such that $\cup_{i=1}^\infty K_i=\Omega$. We denote \[\wt\mu_i=T_i(h\mathds{1}_{K_i})\gam\quad\text{for every }\ i\in\N\] and notice that $\{\wt\mu_i\}_i$ is an increasing sequence of positive measures in $(W^{1,{\vp(\cdot)}}_0(\Omega))'$ with  supports compact in $\Omega$.  Set \[\mu_0 =\wt\mu_0\quad\text{ and }\quad \mu_i=\wt\mu_i-\wt\mu_{i-1}\quad\text{for every }\ i\in\N.\] Then $\sum_{i=1}^k\mu_i=T_k(h\mathds{1}_{K_k})\gam\in\Mb(\Omega)$. Since $\mu_i\geq 0$, we have $\sum_{i=1}^\infty \|\mu_i\|_{\mathcal{M}(\Omega)}<\infty$. Furthermore, $\mup=\sum_{i=1}^\infty\mu_i$ and this series is absolutely convergent in $\Mb(\Omega)$.

\medskip

\noindent {\em Step 3. Construction of decomposition. }\\ We take a standard mollifier $\vr\in C_0^\infty(B(0,1))$ (with $\int_\rn \vr\,dx=1$) and $\vr_k(x) = k^n\vr(kx)$ for every $x\in\rn$,   and consider mollification \[ \mu_{i,k}^\vr(x)=\int_\rn \vr_k(x-y)\,d\mu_i(y )\]
For $k$ large enough, we can decompose $\mu_i=f_i+g_i$  with \[ f_i=\mu_{i,k_i}^\vr\in C_0^\infty(\Omega)\qquad \text{and}\qquad w_i=\mu_i-\mu_{i,k_i}^\vr\in (W^{1,{\vp(\cdot)}}_0(\Omega))'\] 
by choosing for every $i$ sufficiently large $k_0^i$ such that for $k_i>k_0^i$,  $\mu_{i,k_i}^\vr$ belongs to $C_0^\infty(\Omega)$,  so we restrict our attention to such~$k_i$.  Therefore,  we get -- up to a subsequence -- convergence of $\{f_{i}\}_i=\{\mu_{i,k_i}^\vr\}_i$ to $\mup$ in measure and for every $i$ we have $\|f_i\|_{L^1 (\Omega)}\leq \|\mup\|_{\Mb(\Omega)}.$ By Step 2 the series $\sum_{i=1}^\infty f_i$ is convergent in $L^1 (\Omega)$ and there exists its limit $f^0=\sum_{i=1}^\infty f_i\in L^1(\Omega)$. As for convergence of $w_i$ we observe first that due to~\cite[Lemma~6.4.5]{hahabook} we get convergence  of $\{\mu_{i,k}^\vr\}_k$ to $\mu_i$ in $(W^{1,{\vp(\cdot)}}_0(\Omega))'$ as $k\to\infty$. We note that the series $\sum_{i=1}^\infty w_i$ converges in $(W^{1,{\vp(\cdot)}}_0(\Omega))'$ and there exists its limit $w^0=\sum_{i=1}^\infty w_i\in(W^{1,{\vp(\cdot)}}_0(\Omega))'$.  Therefore, the three following series converge in the sense of distributions \[\sum_{i=1}^\infty \mu_i=\mup,\quad \sum_{i=1}^\infty f_i=f^0,\quad \text{and} \quad\sum_{i=1}^\infty w_i=w^0\]  and, consequently, $\mup=f^0+w^0$.

\medskip

\noindent {\em Step 4. Summary. } Let us recall that  the proof starts with justification that for a nonnegative measure $\mu \in L^1(\Omega)+(W^{1,{\vp(\cdot)}}_0(\Omega))'\subset \MP(\Omega) .$ Step~3 provides the reverse implication. 

If the measure was not nonnegative we conduct the above reasoning  on decomposition $\mup=\big((\mup)_++(\mup)_-\big)$ separately for its positive and negative part.  Note that by monotonicity of capacity, if $\CapP(A,\Omega)=0$, then $(\mup)_+(A)=0=(\mup)_-(A)$ ($\vp(\cdot)$-capacity can be achieved over Borel sets included in $A$), see~\cite{bahaha}. Clearly, wherever $\mu_{\vp(\cdot)}$ is positive, so is $f$. 

 Thus also for a signed measure $\mu\in\Mb(\Omega)$ we get that $\mu\in \MP(\Omega)$ if and only if $\mu \in\big( L^1(\Omega)+(W^{1,{\vp(\cdot)}}_0(\Omega))'\big)\cap\Mb(\Omega)$, that is when there exists $f\in L^1(\Omega)$ and $G\in \big(L^{\wt\vp(\cdot)}(\Omega)\big)^n$, such that 
\[\mup=f-\dv G\ \text{ in the sense of distributions}.\] Hence, the proof of the capacitary characterization is completed.
\end{proof}

To conclude Remark~\ref{rem:decomp} we need the following decomposition lemma. Its proof is essentially the one of \cite[Lemma~2.1]{FST}, but we find it valuable to present it for the sake of completeness.
\begin{lemma}\label{lem:basic-decomp}
Suppose $\Omega$ is a bounded set in $\rn$ and $\mathfrak{M}$ is a family of its measurable subsets. Then for every $\mu\in\Mb(\Omega)$ there exist unique decomposition $\mu=\mu_{ac}+\mus$, such that \begin{itemize}
\item[(a)] $\mu_{ac}(D)=0$ for all sets $D\subset\mathfrak{M}$ with $\capP(D)=0$,
\item[(b)] $\mus=\mu\mathds{1}_{N}$ for some set $N\subset\mathfrak{M}$ with $\capP(N)=0$,
\end{itemize}
\end{lemma}
\begin{proof} We shall construct our decomposition making use of an arbitrary fixed sequence of sets $D_1\subset D_2\subset\dots \subset\mathfrak{M}$ of sets with $\capP(D_i)=0$, such that  
\[\lim_{i\to\infty}\mu(D_i)=\alpha:=\sup\{\mu(D):\ D\in\mathfrak{M}\ \text{ and }\ \capP(D)=0\}<\infty.\]Set $D_\infty=\bigcup_{i=1}^\infty D_i$ and notice that $D_\infty\in\mathfrak{M}$, $\capP(D_\infty)=0$ and $\mu(D_\infty)=\alpha$. Then $\mu(D\setminus D_\infty)=0$ for every $D\in\mathfrak{M}$ with $\capP(D)=0$. By defining \[\mu_{ac}=\mathds{1}_{\rn\setminus D_\infty}\mu\qquad\text{and}\qquad \mus=\mathds{1}_{D_\infty}\mu\] we get the decomposition of the desired properties. In particular, uniqueness of the decomposition is evident.
\end{proof}

\section{Approximate problems }\label{sec:appr}
 This section is devoted to analysis of approximate problems with general datum. 

\medskip


In the case of a measure $\mu\in\Mb(\Omega)$ we shall consider  an approximate sequence of bounded functions $\{\mu^s\}\subset C^\infty(\Omega)$ that converges to $\mu$ weakly-$*$ in the space of measures and satisfies~\eqref{muslim}.  
We study solutions to
\begin{equation}
\label{eq:approx-general}
\begin{cases}
-\dv\A(x,\nabla u_s)=\mu^s&\text{in }\ \Omega,\\
 u_s=0&\text{on }\partial \Omega.\end{cases}
\end{equation}  It is known that there exists at least one distibutional solution $u_s\in W^{1,\vp(\cdot)}(\Omega)$ to~\eqref{eq:approx-general}, see the proof by Galerkin approximation in~\cite[Section~5.1.1]{gw-e} under more general growth conditions embracing our case. In fact, since smooth functions are dense in the space where the solutions live i.e. $ W_0^{1,\vp(\cdot)}(\Omega)$, we can test the equation by the solution itself to get energy estimates and, consequently, the distibutional solutions $u_s$ are weak solutions. 
\begin{rem}Note that requiring regularity of $\mu^s$ is not a restriction. 
 In fact, one actually proceed under sole $\mu^s\in(W^{1,\vp(\cdot)}_0(\Omega))'\cap\Mb(\Omega)$.\end{rem}

\begin{proposition}[Basic a priori estimates]\label{prop:apriori-basic}  Let $\Omega$ be bounded open domain in $\rn$, $\A:\Omega\times \rn\to\rn$ satisfy $(\A 1)$--$(\A 4)$, $\vp\in\Phi_c(\Omega)$ satisfy (aInc)$_p$, (aDec)$_q$, (A0), (A1), and (A2), and $\muvp\in\Mb(\Omega)$. Then, for a weak solutions $u_s$ to~\eqref{eq:approx-general} and $k>1$, we have
\begin{flalign} 
&\int_\Omega\vp(x,|\nabla T_k(u_s)|)\,dx\leq  \bar{c}_1 k\,,\label{apriori1}\\
&\int_\Omega\wt\vp\big(x,|\A(x,\nabla T_k(u_s))|\big)\,dx\leq \bar{c}_2 k\,,\label{apriori2}
\end{flalign}
with constants $\bar{c}_1=2\|\mu\|_{\mathcal{M}_b(\Omega)} /{c_1^\vp},\ \bar{c}_2=c(c_1^\vp,c_2^\vp,c_{\Delta_2}(\vp),q,\|\mu\|_{\mathcal{M}_b(\Omega)}, \|1+\gamma\|_{L^{\wt\vp(\cdot)}(\Omega)})>0$. Furthermore, for some  $\bar{c}_3=\bar{c}_3(\Omega,n)>0$
\begin{equation}
\label{Ah0}
|\{|u_s|\geq k\}|\leq \bar{c}_3 k^{1-p}.
\end{equation}
\end{proposition}
Since constants in the above estimates does not depend on $s$, we can infer what follows.
\begin{remark}\label{rem:boundedness}
Note that within our doubling regime this implies that for any fixed $k>0$ the sequence $\{\nabla T_k(u_s)\}_s$ is uniformly bounded in $L^{\vp(\cdot)}(\Omega)$,  $\{\A(x,\nabla T_k(u_s))\}_s$ is uniformly bounded in $L^{\wt\vp(\cdot)}(\Omega)$ and the set $\{|u_s|\geq k\}$ for increasing $k$ is shrinking uniformly in $s$.
\end{remark}

\begin{proof}[Proof of Proposition~\ref{prop:apriori-basic}] To get~\eqref{apriori1}, we use first ($\A$2)$_1$, \eqref{eq:approx-general} tested againts $T_k(u_s)\in W^{1,\vp(\cdot)}_0(\Omega)$, and the above remark, in the following way
\begin{flalign*} 
c_1^\vp\int_\Omega\vp(x,|\nabla T_k(u_s)|)\,dx&\leq \int_\Omega \A(x,\nabla T_k(u_s))\cdot \nabla T_k(u_s)\,dx = \int_\Omega \A(x,\nabla u_s)\cdot \nabla T_k(u_s)\,dx\\
&=\int_\Omega T_k(u_s)\,d\mu^s  \leq 2 k\|\mu^s\|_{\mathcal{M}_b(\Omega)}.
\end{flalign*}
We conclude the last inequality above because of the assumed properties of $\mu^s$.

In order to get~\eqref{apriori2}, we use ($\A$2)$_2$, doubling growth~\eqref{doubl-star}, and finally~\eqref{apriori1} to conclude that for any $k>1$ we have
\begin{flalign} 
\nonumber &\int_\Omega\wt\vp\big(x,|\A(x,\nabla T_k(u_s))|\big)\,dx\leq \int_\Omega \wt\vp\Big(x,1+\gamma(x)+\vp(x,|\nabla T_k(u_s)|)/|\nabla T_k(u_s)|\Big)\,dx\\
\nonumber &\qquad\qquad\leq \frac{1}{2} \left(  \int_\Omega \wt\vp\Big(x,2\big(1+\gamma(x)\big)\Big)\,dx+\int_\Omega \wt\vp\Big(x,\vp\big(x,2|\nabla T_k(u_s)|\big)/|\nabla T_k(u_s)|\Big)\,dx\right)\\
&\qquad\qquad\leq c\left(\int_\Omega \wt\vp\big(x,1+\gamma(x)\big)\,dx+\int_\Omega \vp\big(x, |\nabla T_k(u_s)|\big)\,dx\right)\leq c\,k,\label{intwtvp}
\end{flalign}
where $c= c(c_1^\vp,c_2^\vp,c_{\Delta_2}(\vp),q,\|\mu\|_{\mathcal{M}_b(\Omega)}, \|1+\gamma\|_{L^{\wt\vp(\cdot)}(\Omega)})$.

To get~\eqref{Ah0} we start with observing that $
|\{|u_s|\geq k\}|=|\{T_k(|u_s|)\geq k\}|$. Then by Tchebyshev inequality,  Poincar\'e inequality,  and~\eqref{apriori1} as follows
\begin{flalign*} 
|\{|u_s|\geq k\}|&\leq \int_\Omega\frac{|T_k(u_s)|^p}{k^p}dx\leq \frac{c}{k^p}\int_\Omega |\nabla T_k(u_s)|^p\,dx\\
&\leq \frac{c}{k^p}\int_\Omega \vp(x,|\nabla T_k(u_s)|)\,dx \leq c k^{1-p}\xrightarrow[k\to\infty]{}0.\end{flalign*}\end{proof}

\section{Approximable solutions}\label{sec:solas} Let us find the fundamental properties of limits of approximate problems.
\begin{proposition}[Existence of approximable solutions and convergences]\label{prop:convI} Let $\Omega$ be bounded open domain in $\rn$, $\A:\Omega\times \rn\to\rn$ satisfy $(\A 1)$--$(\A 4)$, $\vp\in\Phi_c(\Omega)$ satisfy (aInc)$_p$, (aDec)$_q$, (A0), (A1), and (A2), and $\mu\in\Mb(\Omega)$. Then there exists at least one approximable solution $u$ (see Definition~\ref{def:sola}). Namely, up to a subsequence 
 $\{u_s\}_{s}$ consisting of solutions to~\eqref{eq:approx-general}, there exists a~function $u\in\mathcal{T}^{1,\vp(\cdot)}_0(\Omega)$, such that when $s\to 0$ and $k>0$ is fixed we have
\begin{flalign} 
\label{conv:ae}
u_s&\to{} u\qquad\quad\text{a.e. in }\ \Omega,\\
\label{conv:strong}T_k(u_s)&\to{}T_k(u)\quad\quad\text{strongly in }\ L^{\vp(\cdot)}(\Omega),\\
\label{conv:weak}\nabla T_k(u_s)&\rightharpoonup{}\nabla T_k(u)\qquad\text{weakly in }\ (L^{\vp(\cdot)}(\Omega))^n,\\
\label{conv:A-ae}\A(x,\nabla T_k(u_s))&\to \A(x,\nabla T_k(u))\quad \text{a.e. in }\ \Omega,\\
\label{conv:A-weak-total}\A(x,\nabla T_k(u_s))&\rightharpoonup{} \A(x,\nabla T_k(u))\quad\text{weakly in }\ (L^{\wt{\vp}(\cdot)}(\Omega))^n\\
\label{conv:A-strong-total-s}\A(x,\nabla T_k (u_s))&\to \A(x,\nabla T_k (u))\quad\text{strongly in  $(L^{1}(\Omega))^n$.}
\end{flalign} 
Moreover, for $k\to\infty$ 
\begin{equation}
\label{conv:A-strong-total-k}\A(x,\nabla T_k (u))\to \A(x,\nabla u)\quad\text{strongly in  $(L^{1}(\Omega))^n$.}
\end{equation} \end{proposition}
\begin{proof}
Having~\eqref{apriori1} we get that $\{T_k(u_s)\}_s$ is uniformly bounded in $W^{1,1}_{0}(\Omega)$. By recalling the Banach-Alaoglu theorem in the reflexive space, we infer that there exists a (non-relabelled) subsequence of $\{u_s\}$ and function $u\in \mathcal{T}^{1,\vp(\cdot)}_0$ such that for $s\to 0$ we have~\eqref{conv:weak}.
Note that in the general case we would have here weak-$\ast$ convergence, but our space is reflexive and these notions of convergences coincide. Since embedding~\eqref{emb} is compact, up to a non-relabelled subsequence, for $s\to 0$ also~\eqref{conv:strong}. Consequently, due to The Dunford-Pettis Theorem, up to an (again) non-relabelled subsequence  $\{T_k(u_s)\}_s$ is a Cauchy sequence in measure and~\eqref{conv:ae} holds. By the same arguments, due to~\eqref{apriori2}, there exists $\mathcal{A}_k^\infty\in (L^{\wt\vp(\cdot)}(\Omega))^n$ such that for $s\to 0$ 
\begin{equation}  
\label{conv:A-weak}\mathcal{A}(x,\nabla T_k(u_s))\rightharpoonup{}\mathcal{A}_k^\infty \quad\text{weakly in }\ (L^{\wt{\vp}(\cdot)}(\Omega))^n\quad \text{for every }k>0.
\end{equation}
The effort will be put now in identification of the limit function
\begin{equation} 
\label{Aik=A} \mathcal{A}_k^\infty = \mathcal{A}(x,\nabla T_k(u))\quad \text{a.e. in $\Omega,\ $ for every }k>0
\end{equation}
and proving that $u$ obtained in this procedure is a very weak solution. Recall that $\mathcal{A}$ is continuous with respect to the last variable, we have the convergence~\eqref{conv:strong} and what remains to prove is fine behaviour of $\{\nabla u_s\}_s$. In order to show that $\{\nabla u_s\}_s$ is a Cauchy sequence in measure we set $\ep >0$ and $m,n\in\N$ arbitrary (large). Given any $t, \tau, r >0$, one has that
\begin{flalign}\label{grad-uk-um}
 |\{ |\nabla u_l-\nabla u_m|>t\}| &\leq  |\{|\nabla u_l|>\tau\}|+|\{|\nabla u_m|>\tau\}|+|\{|u_l|>\tau\}|+|\{|u_m|>\tau\}|\\
 &\ +|\{| u_l - u_m |>r\}|+E\,,\nonumber
\end{flalign}
where
\begin{equation}\label{G}
E=|\{|  u_l - u_m |\leq r,\,|u_l|\leq \tau,\,|u_m|\leq \tau,\,|\nabla u_l|\leq\tau,\,|\nabla u_m|\leq \tau,\,|\nabla u_l-\nabla u_m|>t\}|.
\end{equation}
Note that~\eqref{Ah0} enables to choose for any $\ve>0$ a number $\tau_\ve$ large enough so that for $\tau>\tau_\ve$ we obtain
\begin{equation}\label{grad-uk'n'um-est}|\{|\nabla u_l|>\tau\}|<\ve,\ |\{|\nabla u_m|>\tau\}|<\ve,\ |\{|u_l|>\tau\}|<\ve, \ \text{and}\ |\{|u_m|>\tau\}|<\ve.
\end{equation}
From now on we resrict ourselves to $\tau>\tau_\ve$. On the other hand,  since $\{u_l\}$ is a Cauchy sequence in measure,
\begin{equation}\label{sep52}
|\{| u_l - u_m |>r\}| < \ep\,, \qquad\text{if $l,m,r$ are  sufficiently large.}
\end{equation}
 What remains to prove is that there exists $\delta_{\tau,\ve}>0$, such that for every $\delta<\delta_{\tau,\ve}$, we get\begin{equation}
\label{G-small}|E|<\ve.
\end{equation}
Let us define a set
\[
S=\{(\xi,\eta)\in\rn\times\rn:\ |y|\leq\tau,\ |z|\leq\tau,\ |\xi|\leq\tau,\ |\eta|\leq \tau,\
|\xi-\eta|\geq t\}\,,
\]
which is compact. Consider the function
$\psi:\Omega\to[0,\infty)$ given by
\[\psi(x)=\inf_{ (\xi,\eta)\in S}\left[\left(\A(x,\xi)-\A(x,\eta)\right)\cdot(\xi-\eta)\right].\]
Monotonicity assumption ($\mathcal{A}3$) and the continuity of the
function $\xi\mapsto \A(\cdot,\xi)$ a.e. in~$S$ ensure that $\psi\geq 0$ in $\Omega$. Furthermore,~($\A 4$) implies that
$|\{\A(x,0)=0\}|=0$. Moreover,
\begin{flalign}
\label{psi-a-priori}
\int_G\psi(x)\,dx&\leq \int_G \left(\A(x,\nabla u_l)-\A(x,\nabla u_m)\right)\cdot (\nabla u_l-\nabla u_m)\,dx\\ \nonumber
&\leq \int_{\{|u_l-u_m|\leq r\}} \left(\A(x,\nabla u_l)-\A(x,\nabla u_m)\right) \cdot (\nabla u_l-\nabla u_m)\,dx\\ \nonumber
&= \int_{\Omega} \left(\A(x,\nabla u_l)-\A(x,\nabla u_m)\right)\cdot (\nabla T_r( u_l- u_m))\,dx\\ \nonumber 
&=\int_{\Omega}  T_r( u_l- u_m) \,d\mu^s(x)\ \leq\ 2r\|\mu\|_{\mathcal{M}_b(\Omega)} ,\nonumber
\end{flalign}
where the last but one equality follows on making use of the test function  $T_r(u_l-u_m) $  and in~the~corresponding equation with $l$ replaced by $m$, and subtracting the resultant equations.   Esimate~\eqref{psi-a-priori} and the properties of the function $\psi$ ensure that, if $s$ is chosen sufficiently small, then~\eqref{G-small} holds. From inequalities \eqref{grad-uk-um}, \eqref{grad-uk'n'um-est}, \eqref{G-small},  and \eqref{sep52}, we infer that $\{\nabla u_s\}_s$ is a Cauchy sequence in measure. 

To conclude that that the function $u$ obtained in~\eqref{conv:weak} and~\eqref{conv:strong} is a desired approximable solution, we observe that it belongs to the class $\mathcal T ^{1,\vp(\cdot)}_0(\Omega)$, and that $\nabla u_s \to \nabla u$ $\hbox{a.e. in $\Omega$}$ (up to subsequences), where $\nabla u$ is understood in the sense of~\eqref{gengrad}.
Since $\{\nabla u_s\}$ is a Cauchy sequence in measure, there exist a subsequence (still indexed by $s$) and a~measurable function
$W:\Omega\to\rn$ such that $\nabla u_s\to  W$ a.e. in $\Omega.$ To motivate that $\nabla u = W$ and
\begin{equation}\label{sep63}
\chi_{\{|u|<k\}}  W \in (L^{\vp(\cdot)}(\Omega))^n \quad \hbox{for every $k>0$}
\end{equation}
it suffices to recall~\eqref{conv:weak}. Indeed, then for each fixed $k>0$, there exists a subsequence of $\{u_s\}$, still indexed by $s$, such that
\begin{equation}\label{sep60}
\lim _{s \to \infty} \nabla T_k(u_s)= \lim _{s \to \infty}
\chi_{\{|u_s|<k\}} \nabla u_s = \chi_{\{|u|<k\}} W \quad \hbox{a.e.
in $\Omega$,}
\end{equation}
and $ \lim_{s \to \infty} \nabla T_k(u_s) = \nabla T_k(u)$ 
weakly in $(L^{\vp(\cdot)}(\Omega))^n$. Therefore,  $\nabla T_k(u) = \chi_{\{|u|<k\}} W$ a.e. in $\Omega$, 
whence~\eqref{sep63} follows. Then, due to ($\A 1$) also~\eqref{Aik=A} holds, that is we have~\eqref{conv:A-ae} and~\eqref{conv:A-weak-total}. 
Due to Remark~\ref{rem:A-int} we get uniform integrability of $\{\A(x,\nabla T_k u)\}_k$, so Lebesgue's monotone convergence theorem justifies~\eqref{conv:A-strong-total-k}, where the limit is in $(L^1(\Omega))^n$ by Lemma~\ref{lem:summability}. By~\eqref{apriori2} and Vitali's convergence theorem we infer~\eqref{conv:A-strong-total-s}. \end{proof}

\section{Renormalized solutions}\label{sec:rs}

Our aim now is to analyse the measures generated by truncations of approximable solutions. 

\begin{proposition}\label{prop:sola-trunc-meas}
If $u$ is an approximable solution under assumptions of Proposition~\ref{prop:convI} and $\oppA$ is given by~\eqref{oppA}, then for every $k>0$ we have $\lambda_k:=\oppA(T_ku)\in \Mb(\Omega)\cap (W^{1,\vp(\cdot)}_0(\Omega))'$ and
\begin{equation}
\label{oppA-meas}\int_{\{|u|<k\}} \A(x,\nabla u)\cdot \nabla \phi\,dx=\int_\Omega \phi\,d\lambda_k\quad\text{for every }\ \phi\in W^{1,\vp(\cdot)}_0(\Omega)\cap L^\infty(\Omega).
\end{equation}
Then  for $k\to \infty$ we have \begin{equation}
\label{oppA2}\oppA(T_ku)\rightharpoonup{} \oppA(u)\qquad\text{ weakly-$\ast$ in the space of measures}.
\end{equation}
Moreover,  for every $k>0$ it holds $|\oppA(T_ku)|(\{|u|>k\})=0$ and for every $\phi\in C_0(\Omega)$ we have\begin{flalign}
\label{limplus}\lim_{\delta\to 0^+}&\frac 1\delta\int_{\{k-\delta\leq u\leq k\}}\A(x,\nabla u)\cdot\nabla u\,\phi\,dx=\int\phi\,d\nu^+_k,\\
\label{limminus}\lim_{\delta\to 0^+}&\frac 1\delta\int_{\{-k\leq u\leq -k+\delta\}}\A(x,\nabla u)\cdot\nabla u\,\phi\,dx=\int\phi\,d\nu^-_k\,
\end{flalign}
with $\nu^+_k=\oppA(T_ku)\mres\{u=k\}$ and $\nu^-_k=\oppA(T_ku)\mres\{u=-k\}$.
\end{proposition}

\begin{proof} We prove first weak-$\ast$ convergence of measures generated by truncations of solutions and then their further properties.

\medskip

\noindent {\em {\bf Step 1.}  $\lambda_k\in\Mb(\Omega)\ $ and $\ \oppA(T_ku)\rightharpoonup{} \oppA(u)$  weakly-$\ast$ in the space of measures. }\\ For $k>\delta>0$ we define a Lipschitz functions $h_\delta,\sigma_\delta^+,\sigma_\delta^-:\R\to\R$ satisfying \[\begin{cases}h_\delta(r)=1&\ \text{ if }|r|\leq k-\delta,\\
|h_\delta'(r)|=\tfrac 1\delta&\ \text{ if }k-\delta\leq |r|\leq k,\\
h_\delta(r)=0&\ \text{ if }|r|\geq k.
\end{cases}\qquad\qquad 
\begin{cases}\sigma_\delta^+(r)=0&\ \text{ if }r\leq k-\delta,\\
(\sigma_\delta^+)'(r)=\tfrac 1\delta&\ \text{ if }k-\delta\leq r\leq k,\\
\sigma_\delta^+(r)=1&\ \text{ if }r\geq k,
\end{cases}
\]
and $\sigma_\delta^-(r)=\sigma_\delta^+(-r)$. We note that if $\{u_s\}$ is an approximate sequence from Definition~\ref{def:sola} solving~\eqref{eq:approx-general} with $\mu^s$ being bounded and smooth function, $\phi\in C_0^\infty(\Omega)$, then $h_\delta(u_s)\phi,\sigma_\delta^+(u_s),\sigma_\delta^-(u_s)\in W^{1,\vp(\cdot)}_0(\Omega)\cap L^\infty(\Omega)$ are admissible test functions in~\eqref{eq:sola-decomp}. By testing~\eqref{eq:approx-general} against $h_\delta(u_s)\phi$ with  we get\begin{flalign*}
\int_\Omega h_\delta(u_s)\,\A(x,\nabla u_s)\cdot \nabla\phi\,dx &=\int_\Omega \phi\mu^s h_\delta(u_s)\,dx-\int_\Omega h'_\delta(u_s)\,\A(x,\nabla u_s)\cdot \nabla u_s\,\phi\,dx\\
&=\int_\Omega \phi\,d\lambda_\delta^s+\int_\Omega \phi\,d\gamma_\delta^{s,+}-\int_\Omega \phi\,d\gamma_\delta^{s,-},
\end{flalign*} where\begin{flalign}
\lambda_\delta^s&=\mu^sh_\delta(u_s),\nonumber\\
\nu_\delta^{s,+}&=\tfrac 1\delta\mathds{1}_{\{k-\delta\leq u_s\leq k\}}\A(x,\nabla u_s)\cdot\nabla u_s,\label{nup}\\
\nu_\delta^{s,-}&=\tfrac 1\delta\mathds{1}_{\{- k\leq u_s\leq -k+\delta\}}\A(x,\nabla u_s)\cdot\nabla u_s.\label{num}
\end{flalign}
Observe that \begin{equation*}
\lambda_\delta^s,\nu_\delta^{s,+},\nu_\delta^{s,-}\in L^1(\Omega).
\end{equation*} Indeed,
\[\|\lambda_\delta^s\|_{L^1(\Omega)}\leq\int_\Omega|\mu^s|\,|h_\delta(u_s)|\,dx\leq \int_\Omega|\mu^s| \,dx\leq 2\|\mu\|_{\Mb(\Omega)}.\] 
To estimate $\|\nu_\delta^{s,+}\|_{L^1(\Omega)}$ and $\|\nu_\delta^{s,-}\|_{L^1(\Omega)}$, we test~\eqref{eq:approx-general} against $\sigma_\delta^+(u_s)$ (respectively $\sigma_\delta^-(u_s)$) and obtain
\begin{flalign}
\|\nu_\delta^{s,+}\|_{L^1(\Omega)}&\leq \frac 1\delta \int_{\{k-\delta\leq u_s\leq k\}}\A(x,\nabla u_s)\cdot\nabla u_s\,dx=\int_\Omega \mu^s\sigma^+_\delta(u_s)\,dx\leq 2\|\mu\|_{\Mb(\Omega)},\label{nup2}\\
\|\nu_\delta^{s,-}\|_{L^1(\Omega)}&\leq \frac 1\delta \int_{\{-k\leq u_s\leq -k+\delta\}}\A(x,\nabla u_s)\cdot\nabla u_s\,dx=\int_\Omega \mu^s\sigma^-_\delta(u_s)\,dx\leq 2\|\mu\|_{\Mb(\Omega)}.\label{num2}\end{flalign}
In the end we have that
\[\left\|-\dv \Big(h_\delta(u_s)\,\A(x,\nabla u_s)\Big)\right\|_{L^1(\Omega)}\leq 6\|\mu\|_{\Mb(\Omega)}.\]
Due to Remark~\ref{rem:A-int} and ($\A$2) we get uniform integrability of $\{\A(x,\nabla (T_k (u_s)))\}_k$, so Lebesgue's monotone convergence theorem justifies we can let $\delta\to 0$ getting 
\[|h_\delta(u_s)\,\A(x,\nabla u_s)| \to |\A(x,\nabla( T_k( u_s)))|\quad \text{strongly in }\ L^{1}(\Omega).\]
Therefore $\oppA(T_ku_s)\in\Mb(\Omega)$ and $\|\oppA(T_ku_s) \|_{\Mb(\Omega)}\leq 6\|\mu\|_{\Mb(\Omega)}$, where the bound is uniform with respect to $s$ and $k$. Consequently, the use of Proposition~\ref{prop:convI} enables to infer that also that $\oppA(T_ku)\in\Mb(\Omega)$, $\|\oppA(T_ku) \|_{\Mb(\Omega)}\leq 6\|\mu\|_{\Mb(\Omega)}$, and -- finally -- \eqref{oppA2}. By Remark~\ref{rem:top} we can extend the family of admissible test functions to get \eqref{oppA-meas} and conclusion that  $\oppA(T_ku)\in\Mb(\Omega)\cap (W^{1,\vp(\cdot)}_0(\Omega))'$.

\medskip

\noindent {\em {\bf Step 2.} Existence of a diffuse measure $\vartheta\in\MP(\Omega)$, such that\[\text{ $\vartheta\mres\{|u|<k\}= \oppA(T_\mathlcal{l}u)\mres\{|u|<k\}\quad $ for every $k>0$ and every $\el\geq k$.} \]}
Lemma~\ref{lem:summability} ensures that $\phi\in W^{1,\vp(\cdot)}_0(\Omega)\cap L^\infty(\Omega)$ belongs to $L^1(\Omega,\vartheta)$ with any $\vartheta\in\MP(\Omega)$.

Note that $\lambda_\el\mres\{|u|<k\}=\lambda_k\mres\{|u|<k\}$ for every $\el\geq k>0$. Since the set $\{|u|<k\}$ is $\capP$-quasi open, Lemmas~\ref{lem:qc-open} and~\ref{lem:exh} ensure that there exists an increasing sequence $\{w_j\}$ of nonnegative functions in $W^{1,\vp(\cdot)}_0(\Omega)$ which converges to $\mathds{1}_{\{|u|<k\}}$ $\capP$-quasi everywhere in $\Omega$. Then $w_j=0$ a.e. in $\{|u|\geq k\}$. If $\psi\in C_0^\infty$, then $\phi=w_j\psi\in W^{1,\vp(\cdot)}_0(\Omega)\cap L^\infty(\Omega)$ is an admissible test function in \eqref{oppA-meas}, so  for every $\el\geq k$ we get
\[\int_\Omega w_j\psi\,d\lambda_k=\int_{\{|u|\leq k\}} \A(x,\nabla u)\cdot\nabla( w_j\psi)\,dx=\int_{\{|u|\leq \el\}} \A(x,\nabla u)\cdot\nabla( w_j\psi)\,dx=\int_\Omega w_j\psi\,d\lambda_\el.\]
Passing to the limit with $j\to\infty$ we get 
\[\int_{\{|u|< k\}} \psi\,d\lambda_k= \int_{\{|u|< k\}} \psi\,d\lambda_\el\quad\text{ for every $\ \psi\in C_0^\infty$},\]
so of course $\lambda_\el\mres\{|u|<k\}=\lambda_k\mres\{|u|<k\}$. Consequently, there exists a unique Borel measure $\vartheta$, such that $\vartheta\mres\{|u|=+\infty\}=0$ and $\vartheta\mres\{|u|<k\}=\lambda_\el\mres\{|u|<k\}$ for every $k>0$ and every $\el\geq k$. As $\lambda_k$ vanishes on every set of zero capacity $\capP$, so does $\vartheta$. By~\eqref{oppA2} the measures $|\lambda_k|$ are uniformly bounded with respect to $k$, so $\{|\vartheta|(\{|u|<k\})\}_k$ is bounded. In turn $|\vartheta|(\Omega)<\infty$ and -- finally -- we infer that $\vartheta\in\MP(\Omega)$.

\medskip

\noindent {\em {\bf Step 3.} $\oppA(T_ku)\mres\{|u|>k\}=0$ . } \\
Lemma~\eqref{lem:sola-qc} gives that $u$ is $\capP$-quasicontinuous, thus the set $\{|u|>k\}$ is $\capP$-quasi open. Fix arbitrary open $V\subset\Omega$. By Lemma~\ref{lem:exh}, there exists an increasing sequence $\{\widehat w_j\}$ of nonnegative functions in $W^{1,\vp(\cdot)}_0(\Omega)$ which converges to $\mathds{1}_{V\cap\{|u|<k\}}$ $\capP$-quasi everywhere in $\Omega$. Then $\widehat w_j=0$ a.e. in $\{|u|\leq k\}$ and we can test  \eqref{oppA-meas} against $ w_j \in W^{1,\vp(\cdot)}_0(\Omega)\cap L^\infty(\Omega)$. We obtain
\[\int_\Omega w_j\,d\lambda_k=\int_{\{|u|\leq k\}} \A(x,\nabla u)\cdot\nabla( w_j )\,dx=0.\]
Letting $j\to\infty$ we get that $(\lambda_k\mres \{|u|>k\})(V)=0$. Since $V$ was arbitrary open set, we have what was claimed.

\medskip

\noindent {\em {\bf Step 4.} Limits. } 
 Since we have~\eqref{nup2} and~\eqref{num2}, we get \eqref{limplus} and \eqref{limminus} for any $\phi\in C_0(\Omega)$, with some nonnegative $\nu_k^+,\nu_k^-\in \Mb(\Omega)$. They have the form given in the claim, because  $\oppA(T_ku)\in \Mb(\Omega)\cap (W^{1,\vp(\cdot)}_0(\Omega))'$ has properties proven in Steps~3 and~4.
\end{proof}

\begin{proposition}[Existence of renormalized solutions]\label{prop:ex-ren}Let $\Omega$ be bounded open domain in $\rn$, $\A:\Omega\times \rn\to\rn$ satisfy $(\A 1)$--$(\A 4)$, $\vp\in\Phi_c(\Omega)$ satisfy (aInc)$_p$, (aDec)$_q$, (A0), (A1), and (A2), and $\mu\in\Mb(\Omega)$. Then there exists at least one renormalized solution to~\eqref{eq:main} (Definition~\ref{def:rs}).
\end{proposition}
\begin{proof} By Proposition~\ref{prop:convI} there exists an approximable solution $u\in\mathcal{T}^{1,\vp(\cdot)}_0(\Omega)$ to~\eqref{eq:main}. We shall show that actually it is also a renormalized
solution. Due to Proposition~\ref{prop:sola-trunc-meas}, measure  $\mu$ can be seen as the weak-$\ast$ limit of $\{\lambda_k\}$, which are expressed as\[\lambda_k=\oppA(T_ku)=\vartheta\mres\{|u|<k\}+\nu_k^+-\nu_k^-\]
with $\vartheta\in\MP(\Omega)$, $\nu_k^+,\nu_k^-\in \big( \Mb(\Omega)\setminus\MP(\Omega)\big)\cup\{0\}$ being such that $\nu_k^+=\nu_k^+\mres\{u=k\}$ and $\nu_k^-=\nu_k^-\mres\{u=-k\}$.  Given $h\in W^{1,\infty}(\R)$ having $h'$ with compact support, $\phi\in C_0^\infty(\Omega)$, and arbitrary $k>0$, function $h(T_{k+1}(u))\phi\in W^{1,\vp(\cdot)}_0(\Omega)\cap L^\infty(\Omega)$, so we can test the equation~\eqref{oppA-meas} to get
\begin{flalign}\label{line1-s-r}
\int_{\{|u|\leq k\}}\A(x,\nabla u)\cdot&\nabla \big(h(T_{k+1}(u))\phi\big)\,dx=
\int_{\{|u|\leq k+1\}}h(u)\phi\,d\lambda_k\\
&=
\int_{\{|u|< k\}}h(u)\phi\,d\vartheta+h(k)\int_\Omega \phi\,d\nu_k^+-h(-k)\int_\Omega \phi\,d\nu_k^-.\label{line2-s-r}
\end{flalign}
We need to justify letting $k\to \infty$. We start with the left-hand side of~\eqref{line1-s-r} by having a look on
\[\A(x,\nabla u)\cdot \nabla \big(h(u)\phi\big)=\A(x,\nabla u)\cdot\nabla u\,(h'(u)\phi)+\A(x,\nabla u)\cdot\nabla \phi\, h(u).\] If we prove that both terms on the right-hand side in the last display are integrable, Lebesgue's dominated convergence theorem will give the desired conclusion. Recall that $u\in\mathcal{T}^{1,\vp(\cdot)}_0(\Omega)$ and satisfy~\eqref{apriori1}, so by Proposition~\ref{prop:apriori-basic} and Lemma~\ref{lem:summability}, $\A(\cdot,\nabla u)\in (L^1(\Omega))^n$. Moreover, $h'$ is bounded and ${\rm supp}\,h'\subset[-M,M]$ for some $M>0$, so  \[\A(\cdot,\nabla u)\cdot\nabla u\,h'(u)=\A(\cdot,\nabla T_Mu)\cdot\nabla(T_M u)\,h'(u)\] is integrable by~\eqref{apriori1}. For the second term we see that \[\|\A(x,\nabla u)\cdot\nabla \phi\, h(u)\|_{L^1(\Omega)}\leq \|\A(x,\nabla u)\|_{L^1(\Omega)}\,\|\nabla \phi\|_{L^\infty(\Omega)}\,\|h\|_{L^\infty(\Omega)},\] so it suffices to use the same arguments as before. Therefore~\eqref{line1-s-r} becomes the left-hand side of \eqref{eq:renorm-decomp} in the limit.  By Remark~\ref{rem:decomp} the following decomposition 
\[\mu=\mup+\musp-\musm,\qquad \mup\in\MP(\Omega),\quad 0\leq \musp,\musm\in\big( \Mb(\Omega)\setminus\MP(\Omega)\big)\cup\{0\}\]
is unique. By~\eqref{oppA2} it holds that $\vartheta\mres\{|u|<k\}\rightharpoonup\oppA(u)$. Note that it is also~\eqref{oppA2} to justify testing against $W^{1,\vp(\cdot)}_0(\Omega)\cap L^\infty(\Omega)$-function. To conclude we use Lebesgue's dominated convergence theorem  in~\eqref{line2-s-r}. To motivate the convergence of the first term we note that we can split the first term to positive and negative part, whose majorants are integrable due to Lemma~\ref{lem:summability}. For the remaining two terms it suffices to recall that $h$ is bounded and constant in infinities. By~\eqref{limplus} one has $\nu_k^+\rightharpoonup\musp$ with ${\rm supp}\,\musp\subset\cap_{k>0}\{u>k\}$, and by~\eqref{limminus} also $\nu_k^-\rightharpoonup\musm$ with ${\rm supp}\,\musm\subset\cap_{k>0}\{u<-k\}$. 
\end{proof}

\section{Uniqueness in problems with diffuse measure data }\label{sec:uniq} The previous results worked for a general measure data problems. Here we restrict to diffuse measures to provide uniqueness.


\begin{proposition}[Uniqueness of approximable solutions]\label{prop:uniq-sola}
Under assumptions of Proposition~\ref{prop:convI}, if $\muvp\in\MP(\Omega)$ and $v^j,$ $j=1,2$, are approximable solutions to~\eqref{eq:main} with $\mu_{\vp(\cdot)}\in\MP(\Omega)$, i.e. $v^1,v^2$ satisfy~\eqref{eq:sola-decomp} with $f,G$ as in Definition~\ref{def:sola}, then $v^1=v^2$ a.e. in $\Omega$.
\end{proposition} \begin{proof}
 We suppose $v^1$ and $v^2$ are  solutions obtained as limits of different approximate problems and prove that they have to be equal almost everywhere. By Theorem~\ref{theo:decomp} for every $\mu\in\MP(\Omega)$ there exist $f\in L^1(\Omega)$ and $G\in (L^{\wt\vp(\cdot)}(\Omega))^n$, such that $\mu=f-\dv G$ in the sense of distributions. Using notation from~\eqref{eq:approx-general}, without loss of the generality we can assume that $f,G$ are obtained as limits of approximate sequences $\{f^i_s\}$ in 
$C_0^\infty(\Omega)$, $i=1,2$, satisfying  
\begin{equation}
\label{fs} \text{$f^i_s\to f$ in $L^1(\Omega)\qquad$ and $\qquad \|f^i_s\|_{L^1(E)}\nearrow\|f\|_{L^1(E)}\ \ $ for measurable $\ \ E\subset\Omega$}\end{equation} and  $\{ {G^i_s}\}$ in 
$C_0^\infty(\Omega)$, $i=1,2$,  such that 
\begin{equation}
\label{Gs} G^i_s\to G\ \text{ strongly in }\ (L^{\wt\vp(\cdot)}(\Omega))^n\qquad\text{and}\qquad \vr_{\wt\vp, E}(|G^i_s|)\leq 2 \vr_{\wt\vp, E}(|G|)
\end{equation} on measurable $E\subset\Omega$. Of course then \[\mu^{i,s}=f^i_s-\dv G^i_s\rightharpoonup\mu\qquad\text{ weakly-$*$ in the space of measures}.\] Within this choice of $f^i_s$ and $G^i_s$ we consider the approximate problems  
\begin{equation}
\label{eq:approx-diffuse}\begin{cases}
-\dv\A(x,\nabla v^i_s) =f^i_s-\dv G^i_s &\text{in }\ \Omega,\\
 v^i_s=0&\text{on }\partial \Omega\end{cases}
\end{equation}
and the approximable solution $v^i$ is defined as an a.e. limit of such weak solutions $v^i_s$. The aim is to show that $v^1= v^2$. We fix arbitrary $t,l>0$, use $\phi=T_t(T_l(v^1_s)-T_l(v^2_s))\in W_0^{1,\vp(\cdot)}(\Omega)\cap L^\infty(\Omega)$  as a test function in both~\eqref{eq:sola-decomp} and subtract the equations to obtain for every $s>0$\begin{flalign}\nonumber
L_s=&\int_{\{|T_l(v^1_s)-T_l(v^2_s)|\leq t\}}(\A(x,\nabla v^1_s)- \A(x,\nabla v^2_s) )\cdot( \nabla v^1_s-\nabla v^2_s)\,dx\\
&=\int_\Omega (f^1_s-f^2_s)T_t(T_l(v^1_s)-T_l(v^2_s))\,dx+\int_\Omega (G^1_s-G^2_s)\cdot\nabla T_t(T_l(v^1_s)-T_l(v^2_s))\,dx\ =\ R_s^1+R_s^2.
\label{diff:u-bu}
\end{flalign} The right-hand side above tends to $0$. Indeed, the convergence of $R_s^1$ holds because $|T_t(T_lv^1_s-T_lv^2_s)|\leq t$ and for $s\to 0$ we have $ f^1_s-f^2_s\to 0$ in $L^1(\Omega)$. As for $R_s^2$ it suffices to note that
\begin{flalign*}|R_s^2|&=\left|\int_{\{|T_l(v^1_s)-T_l(v^2_s)|\leq t\}}(G^1_s-G^2_s)\cdot\nabla T_lv^1_s\,dx-\int_{\{|T_l(v^1_s)-T_l(v^2_s)|\leq t\}}(G^1_s-G^2_s)\cdot\nabla T_l(v^2_s)\,dx\right|\\
&\leq \left|\int_\Omega(G^1_s-G^2_s)\cdot\nabla T_l(v^1_s)\,dx\right|+\left|\int_\Omega (G^1_s-G^2_s)\cdot\nabla T_l(v^2_s) \,dx\right|\\
&\leq 2\| G^1_s-G^2_s\|_{L^{\wt\vp(\cdot)}(\Omega)}\|\nabla T_l(v^1_s)\|_{L^{\vp(\cdot)}(\Omega)}+2\| G^1_s-G^2_s\|_{L^{\wt\vp(\cdot)}(\Omega)}\|\nabla T_l(v^2_s)\|_{L^{\vp(\cdot)}(\Omega)}\\
&\leq c\| G^1_s-G^2_s\|_{L^{\wt\vp(\cdot)}(\Omega)},
\end{flalign*} where we used that weak convergence of the $\{\nabla T_l(v^j_s)\}_s$ $(j=1,2)$ in $(L^{\vp(\cdot)}(\Omega))^n$, which in particular implies uniform boundedness of $\{\|\nabla T_l(v^j_s)\|_{L^{\vp(\cdot)}(\Omega)}\}_s$ $(j=1,2)$ and recalled that the strong convergence of $(G^1_s-G^2_s)\to 0$ in $L^{\wt\vp(\cdot)}(\Omega)$. The left-hand side of~\eqref{diff:u-bu}  is nonnegative due to the monotonicity of $\A$ 
 Moreover, as $R_s^1+R_s^2\to 0$, we get
\begin{flalign*}0\leq &
\int_{\{|T_lv^1-T_lv^2 |\leq t\}}(\A(x,\nabla v^1)-\A(x,\nabla v^2 ))\cdot( \nabla v^1 -\nabla v^2 )\,dx\\
\leq& \limsup_{s\to 0}L_s= \limsup_{s\to 0}\,(R_s^1+R_s^2)= 0.\end{flalign*} 
Consequently,  $\nabla v^1 =\nabla v^2$ a.e. in $\{|T_l( v^1) -T_l( v^2) |\leq t\}$ for every $t,l>0$, and so \begin{equation}
\label{nau=nabu}\nabla v^1 =\nabla v^2\quad\text{ a.e. in }\Omega. 
\end{equation} 
Given the boundary value also $v^1= v^2$ a.e. in $\Omega$. 
\end{proof}

\section{Main proof}\label{sec:main-proof}
\begin{proof}[Proof of Theorem~\ref{theo:main}]Existence of approximable solutions is provided in Proposition~\ref{prop:convI}. Proposition~\ref{prop:ex-ren} yields that an approximable solution is a renormalized solutions. Proposition~\ref{prop:sola-trunc-meas} actually localizes the support of singular measures. Approximable solutions can be achieved from renormalized ones by a choice of $h=T_k$. Uniqueness for problems with diffuse data is given for approximable solutions in Proposition~\ref{prop:uniq-sola}.  \end{proof}


\end{document}